\documentclass[10pt]{article}%
\usepackage{amsmath}
\usepackage{amsfonts}
\usepackage{amssymb}
\usepackage{graphicx}%
\usepackage{mathrsfs}
\usepackage[square, comma, sort&compress, numbers]{natbib}
\usepackage{amssymb, color}
\usepackage[body={15.5cm,21cm}, top=3cm]{geometry}%
\providecommand{\U}[1]{\protect \rule{.1in}{.1in}}

\DeclareMathOperator*{\esssup}{ess\,sup}
\newtheorem{theorem}{Theorem}[section]

\newtheorem{definition}[theorem]{{Definition}}

\newtheorem{lemma}[theorem]{Lemma}

\newtheorem{remark}[theorem]{{Remark}}

\newenvironment{proof}[1][Proof]{\noindent \textbf{#1.} }{\  \rule{0.5em}{0.5em}}

\begin{document}

\title{Quadratic Mean-Field Reflected BSDEs}
\author{ Ying Hu\thanks{Univ. Rennes, CNRS, IRMAR-UMR 6625, F-35000, Rennes, France and School of Mathematical Sciences, Fudan University, Shanghai
		200433, China. ying.hu@univ-rennes1.fr.
		Research   supported
		by Lebesgue Center of Mathematics ``Investissements d'avenir"
		program-ANR-11-LABX-0020-01, by CAESARS-ANR-15-CE05-0024 and by MFG-ANR-16-CE40-0015-01.}
	\and Remi Moreau \thanks{Univ. Rennes, CNRS, IRMAR-UMR 6625, F-35000, Rennes, France. remi.moreau@ens-rennes.fr. }
	\and Falei Wang\thanks{Zhongtai Securities Institute for Financial  Studies and School of Mathematics, Shandong University, Jinan 250100, China.
		flwang2011@gmail.com. Research supported by the Natural Science Foundation of Shandong Province for Excellent Youth Scholars (ZR2021YQ01),  the National Natural Science Foundation of China (Nos. 12171280,  12031009 and 11871458) and the Young Scholars Program of Shandong University.}}
\date{}
\maketitle
\begin{abstract}
In this paper, we analyze  mean-field reflected  backward stochastic differential equations when the driver has quadratic growth in the second unknown $z$. Using linearization technique  and BMO martingale theory,  we first apply fixed point argument  to establish uniqueness and existence result for the case with bounded terminal condition and obstacle. Then, with the help of a $\theta$-method, we develop a  successive  approximation procedure to remove the boundedness condition on the terminal condition and obstacle  when the generator is concave (or convex) with respect to the 2nd unknown $z$.
\end{abstract}

\noindent \textbf{MSC-classification}: 60H10, 60H30

\section{Introduction}

Let $(\Omega, \mathscr{F},\mathbf{P})$ be  a given complete probability space  under which $B$ is a $d$-dimensional standard Brownian motion. Suppose $(\mathscr{F}_t)_{0\leq t\leq T}$ is the natural filtration generated by $B$ augmented by the $\mathbf{P}$-null sets and $\mathcal{P}$  the corresponding sigma algebra of progressive sets of $\Omega\times [0,T]$. This paper is devoted to the study  of the following mean-field type reflected backward stochastic differential equations (BSDEs):
\begin{align}\label{my1}
\begin{cases}
&Y_t=\xi+\int_t^T f(s,Y_s,\mathbf{P}_{Y_s},Z_s)ds-\int_t^T Z_sdB_s+K_T-K_t, \quad 0\le t\le T,\\
& Y_t\geq h(t,Y_t, \mathbf{P}_{Y_t}),\quad \forall t\in [0,T] \,\,\mbox{ and }\,\, \int_0^T (Y_t-h(t,Y_t,\mathbf{P}_{Y_t}))dK_t = 0,
\end{cases}
\end{align}
where $\mathbf{P}_{Y_t}$
is the marginal probability distribution of the process $Y$ at time $t$, the terminal condition $\xi$ is a scalar-valued $\mathscr{F}_T$-measurable random variable, the driver $f: \Omega\times [0,T]\times \mathbb{R}\times \mathcal{P}_1(\mathbb{R})\times \mathbb{R}^d\to \mathbb{R}$ and  the constraint $h:\Omega\times[0,T]\times\mathcal{P}_1(\mathbb{R})\times\mathbb{R} \to \mathbb{R}$ are progressively measurable maps with respect to $\mathcal{P}\times \mathcal{B}(\mathbb{R}) \times \mathcal{B}(\mathcal{P}_1(\mathbb{R}))\times \mathcal{B}(\mathbb{R}^{d})$ and $\mathcal{P}\times\mathcal{B}(\mathcal{P}_1(\mathbb{R}))\times\mathcal{B}(\mathbb{R})$ respectively.\medskip

It is well known that El Karoui et al. \cite{EKP} introduced the following  reflected BSDE
\begin{align}\label{my21}
\begin{cases}
&Y_t=\xi+\int_t^T f(s,Y_s,Z_s)ds-\int_t^T Z_sdB_s+K_T-K_t, \quad 0\le t\le T,\\
& Y_t\geq L_t,\quad \forall t\in [0,T] \,\,\mbox{ and }\,\, \int_0^T (Y_t-L_t)dK_t = 0,
\end{cases}
\end{align}
in which the obstacle $L$ is a given stochastic process.  When the  terminal condition is square-integrable  and the driver is Lipschitz in the unknowns $(y,z)$,  the authors of \cite{EKP} obtained the existence and uniqueness of solution to reflected BSDE \eqref{my21} both by a fixed point argument and by penalization method.
Great progress has since then been made in this field, as it has rich connections with obstacle problems of partial differential equations, American option pricing, zero-sum games, switching problems and many others, e.g., see \cite{benezet2021,CK1,karoui1997a, HJ, HZ,HT1} and the references
therein for more details on this topic.
In particular,  the  term $Y$ can be seen as a solution of an optimal stopping problem
\begin{equation}\label{my31}
Y_t=\esssup\limits_{\tau \text{ stopping time }\!\geq\ \! t} \mathbf{E}_t\bigg[\eta\mathbf{1}_{\{\tau=T\}}+L_{\tau}\mathbf{1}_{\{\tau<T\}}+\int^{\tau}_tf(s,Y_s,Z_s)ds\bigg], \ \forall t\leq T.
\end{equation}
\medskip

Recently, in order to study partial hedging of financial derivatives, various mean-field type reflected BSDEs were introduced, in which the driver $f$ and the obstacle $h$ may depend on the law of the term $Y$. For example, Briand, Elie and Hu \cite{BH} considered BSDEs with mean reflection to study the super-hedging problem under running risk management constraint. We refer the reader to \cite{BC1,BE,BC,CH,DD,L} and the references therein for  some other important contributions.\medskip

In particular, motivated by applications in pricing life insurance contracts with surrender options, Djehiche, Elie and   Hamad{\`e}ne formulated in \cite{DES} mean-field reflected BSDEs of the form \eqref{my1}. Under the Lipschitz
hypothesis on the driver, they used a fixed point method to prove the existence and uniqueness result for mean-field reflected BSDEs  \eqref{my1} via the Snell envelope representation \eqref{my31}:
{\small
\begin{align}\label{my41}
\Gamma(U)_t=\esssup\limits_{\tau \text{ stopping time }\!\geq\ \! t} \mathbf{E}_t\bigg[\xi\mathbf{1}_{\{\tau=T\}}+h \left(\tau,U_{\tau},{(\mathbf{P}_{U_s})_{s=\tau}}\right) \mathbf{1}_{\{\tau<T\}}+\int^{\tau}_tf(s,U_s,\mathbf{P}_{U_s})ds\bigg],\, \forall t\leq T,
\end{align}}
in which the driver $f$ is independent of the second unknown $z$.
Indeed, any solution $Y$ to \eqref{my1} is a fixed point of the solution map $\Gamma(U)$.
Note that the comparison principle for mean-filed BSDE is quite restricted, which involves some additional monotone hypothesis on the driver (see \cite{BJ}). Thus, under some additional assumptions, they applied a penalization method to obtain the existence of a solution when the driver $f$ also depends on the second unknown $z$. More precisely, they used a global domination condition in the $z$ component and assumed that
\[f(s,Y_s,\mathbf{P}_{Y_s},Z_s)=F(s,Y_s,\mathbf{E}[Y_s],Z_s), \quad h(t,Y_t, \mathbf{P}_{Y_t})=H(t,Y_t, \mathbf{E}[Y_t]),\]
where $F$ and $H$ are non-decreasing with respect to $\mathbf{E}[Y_s]$.
\medskip

 Recently, Djehiche, Dumitrescu and Zeng  \cite{DDZ} studied the mean-field reflected BSDEs with jumps and  right-continuous and left-limited obstacle.  In the Lipschitz driver case, they  introduced a novel fixed point argument to establish the existence as well as the uniqueness of the solution without these additional assumptions. The main idea is based on the following nonlinear Snell envelope representation  for the reflected BSDE \eqref{my21}:
\begin{align}\label{my51}
Y_t=\esssup\limits_{\tau \text{ stopping time }\!\geq\ \! t} \mathcal{E}_{t,\tau}^{g}[\eta\mathbf{1}_{\{\tau=T\}}+L_{\tau}\mathbf{1}_{\{\tau<T\}}], \ \ \forall t\leq T,
\end{align}
where $\mathcal{E}_{t,\tau}^{g}[\eta\mathbf{1}_{\{\tau=T\}}+L_{\tau}\mathbf{1}_{\{\tau<T\}}]:=y_{t}^{\tau}$ is the solution to the following standard BSDE:
\begin{align} \label{my91} y_{t}^{\tau} = \eta\mathbf{1}_{\{\tau=T\}} + L_{\tau}\mathbf{1}_{\{\tau<T\}} + \int^{\tau}_t f(s,y_{s}^{\tau},z_{s}^{\tau})ds-\int^\tau_tz_{s}^{\tau}dB_s,\end{align}
which does not explicitly involve the term $Z$ and allows to construct a solution map $\Gamma$ when the driver $f$ depends on the second unknown $z$.
Our aim is to establish the existence and uniqueness of the solution to the  mean-field reflected BSDE \eqref{my1} with quadratic driver, i.e.,  the driver $f$ is allowed to have quadratic growth in the second unknown $z$. \medskip

In the BSDEs theory, the research of quadratic case is significantly more difficult than that of Lipschitz case.
Based on the monotone convergence method and PDE-based approximation technique, Kobylanski \cite{K1}   established the solvability of real-valued quadratic BSDEs with bounded  terminal condition. Then, using monotone convergence method and localization stopping times, Briand and Hu \cite{BH2006} extended the existence result of real-valued quadratic BSDEs to the case that the terminal condition can have exponential moment of certain order. Under the additional assumption that the driver $f$  is concave (or convex) with respect to the 2nd unknown  $z$,
Briand and Hu \cite{BH2008} used a $\theta$-method to obtain the uniqueness result. \medskip

With the help of the afore-mentioned results, some
generalizations were obtained for quadratic reflected BSDEs. Indeed, Kobylanski et al. \cite{KL} and  Bayraktar and Yao \cite{BY} made a counterpart study for the case of bounded terminal condition and  obstacle and  of unbounded terminal condition and obstacle, respectively. In particular, they also established the corresponding nonlinear Snell envelope representation \eqref{my51}, which allows us to construct  a iteration map $U\rightarrow\Gamma(U)$ as that of Lipschitz case. \medskip

However, the monotone convergence method for mean-filed BSDE is quite restricted. Thus, we have to develop an alternative approximation approach to obtain a fixed point of the quadratic solution map $\Gamma$.  Note that Tevzadze \cite{Te} proposed a fixed point method for real-valued quadratic BSDEs with bounded terminal condition through  BMO martingale theory. However, when the terminal condition is unbounded, the fixed point argument fails to work since   {the} 2nd unknown $Z$  may be unbounded in the BMO space. Recently, it was found in Fan et al. \cite{FHT2} that $\theta$-method  also provides an approximation procedure for (multi-dimensional) quadratic BSDEs when the driver is concave (or convex) and terminal value has exponential moments of arbitrary order.
For more research on this field, we refer the reader to  \cite{BE0,BE1,CN,HT,L1,HR2016,XZ2016} and the references therein.\medskip

 Thanks to these results, we could show that the quadratic solution map $\Gamma$ admits a unique fixed point by contraction map argument and the $\theta$-method.
When the terminal condition and obstacle are bounded, we  combine
nonlinear  Snell envelope representation and BMO martingale theory to show the quadratic solution map is  {a} contraction. In comparison to  that of \cite{DDZ}, we use linearization technique
to estimate the difference of two solutions instead of It\^{o}'s formula. As a byproduct, our argument removes a  domination condition \cite [Assumption 2.1 (ii)(b)]{DDZ} for the Lipschitz case. \medskip

In the unbounded case, we first apply  nonlinear  Snell envelope representation to introduce an approximation procedure through a sequence of quadratic reflected BSDEs with unbounded terminal condition and obstacle. Then, utilizing quadratic BSDEs theory and the $\theta$-method, we show the convergence of the approximating sequences by some delicate and involved technique computations. In particular, the corresponding limit is the unique solution to the quadratic mean-field reflected BSDE \eqref{my1} with the concave driver and the terminal value of exponential moments of arbitrary order.
In conclusion, we develop quadratic mean-field reflected BSDEs theory which gives some extension of the result from \cite{DES} and \cite{DDZ} to the quadratic case.

\medskip

The paper is organized as follows. In section 2, we start with some technical lemmas and \textcolor{red} {a} revisit to Lipschitz case to illustrate the main idea. Section 3 is devoted to the quadratic case with bounded terminal condition and obstacle, while Section 4 removes the boundedness condition using convexity on the driver.

\subsubsection*{Notation.}
For  each Euclidian space,  we  denote by $\langle\cdot ,\cdot \rangle$  and  $|\cdot|$ its scalar product and the associated norm, respectively.
Then, for each $p\geq 1$, we consider the following collections:

\begin{description}
	\item[$\bullet$] $ \mathcal{L}^{p}$ is the collection of  real-valued $\mathscr{F}_T$-measurable random variables $\xi$ satisfying {\small \[
	\|\xi\|_{\mathcal{L}^{p}}=\mathbf{E}\left[|\xi|^p\right]^{\frac{1}{p}}<\infty;
	\]}
	\item[$\bullet$]  $\mathcal{L}^{\infty}$ is the collection of  real-valued $\mathscr{F}_T$-measurable random variables $\xi$ satisfying 	{\small \[
 \|\xi\|_{\mathcal{L}^{\infty}}=\esssup\limits_{\omega\in\Omega}|\xi(\omega)|<\infty;\]}
	\item[$\bullet$]  $\mathcal{H}^{p,{d}}$  is the collection of $\mathbb{R}^{{d}}$-valued  $\mathscr{F}$-progressively measurable   processes $(z_t)_{0\leq t\leq T}$
	satisfying
	{\small \begin{align*}
	\|z\|_{\mathcal{H}^{p}}=\mathbf{E}\bigg[\bigg(\int^T_0|z_t|^2dt \bigg)^{\frac{p}{2}}\bigg]^{\frac{1}{p}}<\infty;
	\end{align*}}

	\item[$\bullet$]  $\mathcal{S}^p$  is the collection of real-valued  $\mathscr{F}$-adapted continuous  processes $(y_t)_{0\leq t\leq T}$ satisfying
		{\small\begin{align*}
	\|y\|_{\mathcal{S}^{p}}=\mathbf{E}\bigg[\sup_{t\in[0,T]}|y_t|^p\bigg]^{\frac{1}{p}}<\infty;
	\end{align*}}
	
\item[$\bullet$]  $\mathcal{S}^{\infty}$  is the collection of  real-valued  $\mathscr{F}$-adapted continuous  processes $(y_t)_{0\leq t\leq T}$ satisfying
{\small \[
\|y\|_{\mathcal{S}^{\infty}}=\esssup\limits_{(t,\omega)\in[0,T]\times\Omega}|y(t,\omega)|<\infty;
\]}
\item[$\bullet$]  $\mathcal{A}^p$ is the collection of continuous non-decreasing processes $(K_t)_{0\leq t\leq T}\in \mathcal{S}^{p}$ with $K_0=0$;
	
	\item[$\bullet$] $\mathcal{P}_p(\mathbb{R})$ is the collection of all probability measures over $(\mathbb{R},\mathcal{B}(\mathbb{R}))$ with finite $p^{th}$ moment, endowed with the $p$-Wasserstein distance $W_p$;
	\item[$\bullet$] $ \mathbb{L}^{p}$ is the collection of  real-valued $\mathscr{F}_T$-measurable random variables $\xi$ satisfying $\mathbf{E}\left[e^{p|\xi|}\right]<\infty$;
	\item[$\bullet$]   $\mathbb{S}^p$ is the collection
of all stochastic processes $Y$ such that $e^Y\in \mathcal{S}^p$;
\item[$\bullet$] $\mathbb{L}$ is the collection of all random variables $\xi\in\mathbb{L}^p$ for any $p\geq 1$, and   $\mathcal{H}^d$, $\mathcal{A}$ and $\mathbb{S}$ are defined in a similar way;
\item[$\bullet$] $\mathcal{T}_{t}$ is the collection of $[0,T]$-valued $\mathcal{F}$-stopping times  $\tau$  such that $\tau \geq t$ $\mathbf{P}$-a.s.;
\item[$\bullet$]
 $BMO$  is the collection of    $\mathbb{R}^d$-valued progressively measurable  processes $(z_t)_{0\leq t\leq T}$  such that
\begin{align*}
\|z\|_{BMO}:=\sup\limits_{\tau\in\mathcal{T}_0}\esssup\limits_{\omega\in\Omega}\mathbf{E}_{\tau}\bigg[\int^T_{\tau}|z_s|^2ds\bigg]^{\frac{1}{2}} < \infty.
\end{align*}
\end{description}\smallskip
   Denote by $\ell_{[a,b]}$ the corresponding collections for the stochastic processes {with} time indexes on $[a,b]$ for $\ell=\mathcal{H}^{p,d},\mathcal{S}^p,\mathcal{S}^\infty$ and so on. For each $Z\in BMO$, we set
 \[
  \mathscr{Exp}(Z\cdot B)_0^t=\exp\left(\int^t_0 Z_s dB_s-\frac{1}{2}\int^t_0|Z_s|^2ds\right),
 \]
which is a martingale  by
 \cite{K}. Thus it follows from Girsanov's theorem that  $(B_t-\int_{0}^tZ_sds)_{0\leq t\leq T}$
is a Brownian motion under the equivalent probability measure $\mathscr{Exp} (Z\cdot B)_{0}^T d\mathbf{P}$.
\section{A reminder in the Lipschitz case}

\subsection{Preliminaries}

{Let us start by giving the definition of a solution and some technical results, which will be frequently used in our subsequent discussions.}

\begin{definition}
By a solution to \eqref{my1}, we mean a  triple of progressively measurable processes $(Y, Z, K)$  such that \eqref{my1} holds.
\end{definition}

\noindent For each $\mathcal{F}$-stopping time $\tau$ taking values in $[0,T]$ and for every $\mathcal{F}_{\tau}$-measurable function $\eta\in  \mathcal{L}^p$ for some $p>1$, we first define the following $g$-evaluation (see \cite{P1}):
\[
\mathcal{E}_{t,\tau}^{g}[\eta]:=y^{\tau}_t, \ \ \forall t\in [0,T],
\]
where $y^{\tau}$ is the solution to the following  BSDE on the  random time horizon $[0,\tau]$
\begin{align}\label{myq2}
y^{\tau}_t=\eta+\int^{\tau}_t g(s,y^{\tau}_s,z^{\tau}_s)ds-\int^{\tau}_t z^{\tau}_s dB_s.
\end{align}
If the  BSDE \eqref{myq2} admits a unique solution $(Y,Z)\in \mathcal{S}^p\times\mathcal{H}^{p,d}$, then it is easy to check that
\[y^{\tau}_{t}=y^{\tau}_{t\wedge\tau}, z^{\tau}_t=z^{\tau}_t\mathbf{1}_{[0,\tau]}(t), \ \ \forall t\in [0,T].\]

\medskip \noindent Next, we introduce the following reflected BSDE:
\begin{align}\label{my3}
\begin{cases}
&Y_t=\eta+\int_t^T g(s,Y_s,Z_s)ds-\int_t^T Z_sdB_s+K_T-K_t, \quad 0\le t\le T,\\
& Y_t\geq L_t,\quad \forall t\in [0,T] \,\,\mbox{ and }\,\, \int_0^T (Y_t-L_t)dK_t = 0,
\end{cases}
\end{align}
where $\eta\in\mathcal{L}^p$, $L\in\mathcal{S}^p$ with $L_T\leq \eta$.
We assume that the driver $g$ satisfies the following comparison principle.
\begin{description}
\item[Comparison principle:]
For $i=1,2$, let $(y^i,z^i,v^i)\in\mathcal{S}^p\times\mathcal{H}^{p,d}\times\mathcal{A}^p$ be a solution to the following BSDE
\[
y^i=\xi^i+\int^T_tg(s,y^i_s,z^i_s)ds+v^i_T-v^i_t-\int^T_tz^i_sdB_s.
\]
If $\xi^1\geq \xi^2$ and $v^2\equiv 0$, then $y^1_t\geq y^2_t$ for every $t\in[0,T]$.
\end{description}
We have the following nonlinear Snell envelope  representation  for the solution of the reflected BSDE \eqref{my3}.

\begin{lemma}\label{my6} Let $(Y,Z,K)\in \mathcal{S}^p\times\mathcal{H}^{p,d}\times\mathcal{A}^p$ be a solution to the reflected BSDE \eqref{my3}. If $g$ satisfies the  comparison principle, then
\[
Y_t=\esssup_{\tau\in\mathcal{T}_{t}}\mathcal{E}_{t,\tau}^{g}[\eta\mathbf{1}_{\{\tau=T\}}+L_{\tau}\mathbf{1}_{\{\tau<T\}}].
\]
In particular, the reflected BSDE \eqref{my3} has at most one solution.
\end{lemma}
\begin{proof}
For any $\tau\in\mathcal{T}_{t}$, we have
\[
Y_s=Y_{\tau}+\int_s^{\tau} g(r,Y_r,Z_r)ds-\int_s^{\tau} Z_rdB_r+K_{\tau}-K_r, \ \ \forall s\in[t,\tau].
\]
Note that $Y_{\tau}\geq \eta\mathbf{1}_{\{\tau=T\}}+L_{\tau}\mathbf{1}_{\{\tau<T\}}$ and $K$ is a non-decreasing process. It follows from the comparison principle that
\[
Y_t\geq \mathcal{E}_{t,\tau}^{g}[\eta\mathbf{1}_{\{\tau=T\}}+L_{\tau}\mathbf{1}_{\{\tau<T\}}], \ \ \forall\tau\in\mathcal{T}_{t}.\]

\noindent On the other hand, we define the stopping time $\tau^*=\inf\{r\in[t,T]: Y_r=L_r\}\wedge T.$ Since $Y_r\geq L_r$ and $\int_t^T (Y_r-L_r)dK_r = 0$, we conclude that $K_{\tau^*}=K_t$, which indicates that
\[
Y_s=Y_{\tau^*}+\int_s^{\tau^*} g(r,Y_r,Z_r)ds-\int_s^{\tau^*} Z_rdB_r, \ \ \forall s\in[t,\tau^*].
\]
Note that $Y_{\tau^*}=\xi\mathbf{1}_{\{\tau^*=T\}}+L_{\tau^*}\mathbf{1}_{\{\tau^*<T\}}$ by the definition of $\tau^*$. It follows that
\[
Y_t=\mathcal{E}_{t,\tau^*}^{g}[\eta\mathbf{1}_{\{\tau^*=T\}}+L_{\tau^*}\mathbf{1}_{\{\tau^*<T\}}],
\]
which completes the proof.
\end{proof}

\begin{remark}{\upshape
It is obvious that the comparison principle and the nonlinear Snell envelope  representation  hold in the Lipschitz case. We refer to \cite{BY,KL} for some sufficient conditions under which the results hold for the quadratic case.
}
\end{remark}

\begin{remark}\label{my778}{\upshape
It follows from Lemma \ref{my6} that  any solution $Y$ to the {mean-field reflected} BSDE \eqref{my1} is a fixed point of the following map $\Gamma$:
\[
\Gamma(U)_t:=\esssup_{\tau\in\mathcal{T}_{t}} \, \mathcal{E}_{t,\tau}^{f^U}[\xi\mathbf{1}_{\{\tau=T\}}+h \left(\tau,U_{\tau},{(\mathbf{P}_{U_s})_{s=\tau}}\right) \mathbf{1}_{\{\tau<T\}}], \ \ \forall t\in[0,T],
\]
where the   driver {$f^U$ is given by} $f^U(t,z):=f(t,U_t,\mathbf{P}_{U_t},z)$.}
\end{remark}
\noindent The representation given by Lemma \ref{my6} is an important tool for our main results in the following sections. With the help of the representation, we will combine a linearization technique, a fixed point argument and the $\theta$-method to study quadratic mean-field reflected BSDEs.
The authors of \cite{DDZ} applied the representation result and a fixed point argument to prove existence and uniqueness of a solution for mean-field reflected BSDEs with jumps when the driver is Lipschitz. In order to illustrate our main idea and present some  preliminaries involved, we first deal with the Lipschitz case via a linearization technique and a fixed point method. As a byproduct, our argument removes a  domination condition \cite [Assumption 2.1 (ii)(b)]{DDZ}.

\subsection{Revisit to the Lipschitz case}
In what follows,  we make use of the following  conditions on the terminal condition $\xi$, the driver $f$ and the constraint $h$.
\begin{description}
	\item[(B1)] There exists a constant $p>1$ such that $\xi \in \mathcal{L}^p$ with $\xi\geq h(T,\xi,  \mathbf{P}_{\xi})$.
	\item[(B2)] The process $f(t,0,{\delta_0},0)$ belongs to $\mathcal{H}^{p,1}$ and  there exists a constant $\lambda>0$ such that for any $t\in[0,T]$, $y_1,y_2\in \mathbb{R}$, $v_1,v_2\in \mathcal{P}_1(\mathbb{R})$ and $z_1,z_2\in \mathbb{R}^{d}$
\begin{equation*}
|f(t,y_1,v_1,z_1)-f(t,y_2,v_2,z_2)|\leq \lambda \left( |y_1-y_2|+{W_1(v_1,v_2)}+|z_1-z_2| \right).
\end{equation*}
	\item[(B3)] The process $h(t,y,{v})$ belongs to $\mathcal{S}^{p}$ for any {$y\in\mathbb{R}$, $v\in \mathcal{P}_1(\mathbb{R})$} and there exist two constants $\gamma_1,\gamma_2>0$  such that for any  $t\in[0,T]$, { $y_1,y_2\in \mathbb{R}$, $v_1,v_2\in \mathcal{P}_1(\mathbb{R})$}
\begin{equation*}
|h(t,y_1,v_1)-h(t,y_2,v_2)|\leq \gamma_1 |y_1-y_2|+\gamma_2{W_1(v_1,v_2)}.
\end{equation*}
\end{description}

\noindent We are now ready to state the main result of this section.
\begin{theorem}\label{my1161}
Assume that \emph{(B1)}-\emph{(B3)} are satisfied.  If $\gamma_1$ and $\gamma_2$ satisfy   \begin{align}\label{myq9501}
(\gamma_1+\gamma_2)^{\frac{p-1}{p}}\left(\bigg(\frac{p}{p-1}\bigg)^{p}\gamma_1+\gamma_2\right)^{\frac{1}{p}}<1,
\end{align}
then the {mean-field reflected} BSDE \eqref{my1}  admits a unique
 solution $(Y,Z,K)\in \mathcal{S}^{p}\times \mathcal{H}^{p,d}\times\mathcal{A}^p$.
\end{theorem}
\begin{remark}{\upshape
We remove the additional domination condition from \cite [Assumption 2.1 (ii)(b)]{DDZ}, that requires: \[\text{The process $\sup_{(y,v)\in\mathbb{R}\times\mathcal{P}_1(\mathbb{R})}|h(t,y,v)|$ belongs to $\mathcal{S}^{p}$.} \]
}
\end{remark}
\begin{remark}\label{my996}{\upshape
Note that the enhanced sufficient condition \eqref{myq9501} is the same  as in \cite[Theorem 3.1]{DES}. The authors of \cite{DDZ} proved that mean-field reflected BSDEs with jumps admit a unique solution  under the following enhanced sufficient condition \[
2^{p-1}(\gamma_1^p+\gamma_2^p)\leq 1.
\]
When $\gamma_1=0$, it reduces to the condition $\gamma_2\leq 2^{\frac{1}{p}-1}$, whereas \eqref{myq9501}  reduces to the condition $\gamma_2\leq 1$. On the other hand, it is easy to check that
\[
(\gamma_1+\gamma_2)^{p-1}\left(\bigg(\frac{p}{p-1}\bigg)^{p}\gamma_1+\gamma_2\right)\geq 2^{p-1}(\gamma_1^p+\gamma_2^p)
\]
when $\gamma_1=\gamma_2.$
 }
\end{remark}

\noindent We are now ready to prove Theorem \ref{my1161}. More precisely, we first state the existence and uniqueness of the solution on a small time interval $[T-h, T]$, in which $h$ is to be determined later. Then, we stitch the local solutions to build the global solution.\medskip

\noindent According to Assumption (B3), it is obvious that $(h(s,U_s,\mathbf{P}_{U_s}))_{s\in [T-h, T]} \in \mathcal{S}^p_{[T-h, T]}$ for any $U\in \mathcal{S}^p_{ [T-h, T]} $. It follows from Lemma \ref{my6} and \cite[Theorem 3]{HP} that $\Gamma(U)$ is the  $\mathcal{S}^p$-solution to the reflected BSDE \eqref{my3} with data $(\eta,g,L)=(\xi,f^U,h(\cdot,U_\cdot,\mathbf{P}_{U_\cdot}))$. Thus for any $h\in (0,T]$, we have \[\Gamma\left(\mathcal{S}^p_{[T-h, T]} \right)\subset \mathcal{S}^p_{[T-h, T]}.\]

\noindent Let us now show uniqueness and existence of the local solution for the mean-field reflected BSDE \eqref{my1}.

\begin{lemma}\label{my3.5}
Assume that \emph{(B1)}-\emph{(B3)} hold.  If $\gamma_1$ and $\gamma_2$ satisfy \eqref{myq9501}, then there exists a constant $\delta>0$ depending only on $p,\lambda,\gamma_1$ and $\gamma_2$ such that for any $h\in(0,\delta]$, the {mean-field reflected} BSDE \eqref{my1} admits a unique solution $(Y,Z,K)\in \mathcal{S}^{p}_{[T-h,T]}\times \mathcal{H}^{p,d}_{[T-h,T]}\times\mathcal{A}^p_{[T-h,T]}$ on the time interval $[T-h,T]$.
\end{lemma}

\begin{proof} The proof will be divided into three steps.\medskip

\noindent {\bf Step 1 (A priori estimate).} Let $U^i\in \mathcal{S}^p$, $i=1,2$.
It follows from Lemma \ref{my6} that
\begin{align}\label{my1001}
\Gamma(U^i)_t:=\esssup_{\tau\in\mathcal{T}_{t}} y^{i,\tau}_t, \ \ \forall t\in[0,T],
\end{align}
in which $y^{i,\tau}_t$ is the solution of the following  BSDE
 \begin{align}\label{myq528}
y^{i,\tau}_t=\xi\mathbf{1}_{\{\tau=T\}}+h (\tau,U^i_{\tau},(\mathbf{P}_{U^i_s})_{s=\tau}) \mathbf{1}_{\{\tau<T\}}+\int^{\tau}_t f(s,U^{i}_s,\mathbf{P}_{U^i_s},z^{i,\tau}_s)ds-\int^{\tau}_t z^{i,\tau}_s dB_s.
\end{align}
For each $t\in[0,T]$, denote by \[
\beta_t=\frac{f^{U^1}(t,z^{1,\tau}_t)-f^{U^1}(t,z^{2,\tau}_t)}{|z^{1,\tau}_t-z^{2,\tau}_t|^2}(z^{1,\tau}_t-z^{2,\tau}_t)\mathbf{1}_{\{|z^{1,\tau}_t-z^{2,\tau}_t|\neq 0\}}.
\]
Then, the pair of processes $(y^{1,\tau}-y^{2,\tau},z^{1,\tau}-z^{2,\tau})$ solves the following BSDE:
{ \begin{align*}
y^{1,\tau}_t-y^{2,\tau}_t &= h (\tau,U^1_{\tau},(\mathbf{P}_{U^1_s})_{s=\tau}) \mathbf{1}_{\{\tau<T\}}-h (\tau,U^2_{\tau},(\mathbf{P}_{U^2_s})_{s=\tau}) \mathbf{1}_{\{\tau<T\}}-\int^T_t\mathbf{1}_{[0,\tau]}(s)(z^{1,\tau}_s-z^{2,\tau}_s)dB_s\\
&\hspace*{2cm} +\int^T_t\left(\beta_s(z^{1,\tau}_s-z^{2,\tau}_s)^{\top}+f^{U^1}(s,z^{2,\tau}_s)-f^{U^2}(s,z^{2,\tau}_s)\right)\mathbf{1}_{[0,\tau]}(s)ds.
\end{align*}
}

\noindent Note that $\widetilde {B}_t:=B_t-\int_{0}^t\beta^{\top}_s\mathbf{1}_{[0,\tau]}(s)ds$, defines a Brownian motion under the equivalent probability measure $\widetilde{\mathbf{P}}$ given by
$d\widetilde{\mathbf{P}}: = \mathscr{Exp} (\beta\mathbf{1}_{[0,\tau]}\cdot B)_{0}^Td\mathbf{P}$.
It follows that for every $t\in[0,T]$
\begin{align*}
y^{1,\tau}_t-y^{2,\tau}_t
&=\mathbf{E}^{\widetilde{\mathbf{P}}}_t\bigg[h (\tau,U^1_{\tau},(\mathbf{P}_{U^1_s})_{s=\tau}) \mathbf{1}_{\{\tau<T\}}-h (\tau,U^2_{\tau},(\mathbf{P}_{U^2_s})_{s=\tau}) \mathbf{1}_{\{\tau<T\}} \\
&\hspace*{6cm}+\int^{\tau}_t \left(f^{U^1}(s,z^{2,\tau}_s)-f^{U^2}(s,z^{2,\tau}_s)\right)ds\bigg]\\
&=\mathbf{E}_t\bigg[\mathscr{Exp} (\beta\mathbf{1}_{[0,\tau]}\cdot B)_{t}^T\bigg(h (\tau,U^1_{\tau},(\mathbf{P}_{U^1_s})_{s=\tau}) \mathbf{1}_{\{\tau<T\}}-h (\tau,U^2_{\tau},(\mathbf{P}_{U^2_s})_{s=\tau}) \mathbf{1}_{\{\tau<T\}}\\
&\hspace*{6cm} +\int^{\tau}_t \left(f^{U^1}(s,z^{2,\tau}_s)-f^{U^2}(s,z^{2,\tau}_s)\right)ds\bigg)\bigg].
\end{align*}
Noting that  $|\beta_t| \leq \lambda$ and by a standard computation, we have that for any $q\geq 1$,
\[
\mathbf{E}_t\left[ \big|\mathscr{Exp} (\beta\mathbf{1}_{[0,\tau]}\cdot B)_{t}^T\big|^q\right]\leq \exp\bigg(\frac{\lambda^2}{2}(q^2-q)(T-t)\bigg) .
\]
In view of H\"{o}lder's inequality, we have for any $\mu\in (1,p)$ and any $t\in[T-h,T]$,
\begin{align*}
&|y^{1,\tau}_t-y^{2,\tau}_t|\\
&\!\!\leq \exp\bigg(\frac{\lambda^2h}{2(\mu-1)}\bigg)\mathbf{E}_t\bigg[\bigg((\gamma_1+\lambda h)\sup\limits_{s\in[T-h,T]}|U_s^1-U^2_s|+(\gamma_2+\lambda h)\sup\limits_{s\in[T-h,T]}\mathbf{E}[|U_s^1-U^2_s|]\bigg)^\mu\bigg]^{\frac{1}{\mu}},
\end{align*}
which together with  \eqref{my1001}  implies the following, for any $t\in[T-h,T]$
 { \begin{align}\label{myq9601}\begin{split}
&|\Gamma(U^1)_t-\Gamma(U^2)_t|^p\\
&\leq \exp\bigg(\frac{p\lambda^2 h}{2(\mu-1)}\bigg)\mathbf{E}_t\bigg[\bigg((\gamma_1+\lambda h)\sup\limits_{s\in[T-h,T]}|U^1_s-U^2_s|+(\gamma_2+\lambda h)\sup\limits_{s\in[T-h,T]}\mathbf{E}[|U_s^1-U^2_s|]\bigg)^\mu\bigg]^{\frac{p}{\mu}}.
\end{split}
\end{align}
}

\noindent {\bf Step 2 (The contraction).}
 {The convexity inequality} $(ax+by)^\rho\leq (a+b)^{\rho-1}(ax^\rho+by^\rho)$ holds for any non-negative constants $a,b,x$, $y$ and $\rho\geq 1$. It follows that
\begin{align*}
&\mathbf{E}_t\bigg[\bigg((\gamma_1+\lambda h)\sup\limits_{s\in[T-h,T]}|U^1_s-U^2_s|+(\gamma_2+\lambda h)\sup\limits_{s\in[T-h,T]}\mathbf{E}[|U_s^1-U^2_s|]\bigg)^\mu\bigg]^{\frac{p}{\mu}}\\
&\leq (\gamma_1+\gamma_2+2\lambda h)^{\frac{p(\mu-1)}{\mu}}
\bigg((\gamma_1+\lambda h)\mathbf{E}_t\bigg[\sup\limits_{s\in[T-h,T]}|U^1_s-U^2_s|^\mu\bigg]+(\gamma_2+\lambda h)\sup\limits_{s\in[T-h,T]}\mathbf{E}[|U_s^1-U^2_s|]^\mu\bigg)^{\frac{p}{\mu}}\\
&\leq (\gamma_1+\gamma_2+2\lambda h)^{p-1}\bigg((\gamma_1+\lambda h)\mathbf{E}_t\bigg[\sup\limits_{s\in[T-h,T]}|U^1_s-U^2_s|^\mu\bigg]^{\frac{p}{\mu}}+(\gamma_2+\lambda h)\sup\limits_{s\in[T-h,T]}\mathbf{E}[|U_s^1-U^2_s|]^p\bigg).
\end{align*}
Recalling \eqref{myq9601} and
applying {Doob's maximal inequality}, we derive
\begin{align*}
&\mathbf{E}\bigg[\sup\limits_{t\in[T-h,T]}|\Gamma(U^1)_t-\Gamma(U^2)_t|^p\bigg]\leq
\exp\bigg(\frac{p\lambda^2h}{2(\mu-1)}\bigg)(\gamma_1+\gamma_2+2\lambda h)^{p-1}\\
&\hspace*{2cm}\times \bigg((\gamma_1+\lambda h)\bigg(\frac{p}{p-\mu}\bigg)^{\frac{p}{\mu}}\mathbf{E}\bigg[\sup\limits_{s\in[T-h,T]}|U^1_s-U^2_s|^p\bigg]+(\gamma_2+\lambda h)\sup\limits_{s\in[T-h,T]}\mathbf{E}[|U_s^1-U^2_s|]^p\bigg).
\end{align*}

\noindent Consequently, for any $\mu\in(1,p)$ and $h\in(0,(\mu-1)^2]$, we have
 \begin{align*}
\mathbf{E}\bigg[\sup\limits_{t\in[T-h,T]}|\Gamma(U^1)_t-\Gamma(U^2)_t|^p\bigg]^{\frac{1}{p}}\leq \Lambda(\mu)
\mathbf{E}\bigg[\sup\limits_{s\in[T-h,T]}|U^1_s-U^2_s|^p\bigg]^{\frac{1}{p}}
\end{align*}
with
\[
\Lambda(\mu)=\exp\left(\frac{\lambda^2(\mu-1)}{2}\right)(\gamma_1+\gamma_2+2\lambda (\mu-1)^2)^{\frac{p-1}{p}}\bigg((\gamma_1+\lambda(\mu-1)^2)\bigg(\frac{p}{p-\mu}\bigg)^{\frac{p}{\mu}}+(\gamma_2+\lambda (\mu-1)^2)\bigg)^{\frac{1}{p}}.
\]
Under Assumption \eqref{myq9501}, we can then find a small enough constant $\mu^*\in(1,p)$ depending only on $p,\lambda,\gamma_1$ and $\gamma_2$ such that $\Lambda (\mu^*) < 1.$
Let us define \begin{equation}\label{defdelta}
{\delta}:=(\mu^*-1)^2.
\end{equation}
It is now obvious that  $\Gamma$ is a contraction map on the time interval $[T-h,T]$ for any $h\in(0,\delta]$. \medskip

\noindent {\bf Step 3 (Uniqueness and existence).} Note that any solution $Y$ to the {mean-field reflected} BSDE~\eqref{my1} is a fixed point of the map $\Gamma$.
For any $h\in(0,\delta]$, $\Gamma$ has a unique fixed point $Y\in\mathcal{S}^p_{[T-h,T]}$, so that \[
Y_t=\esssup_{\tau\in\mathcal{T}_{t}}\mathcal{E}_{t,\tau}^{f^Y}[\xi\mathbf{1}_{\{\tau=T\}}+h(\tau,Y_{\tau},{(\mathbf{P}_{Y_s})_{s=\tau}})\mathbf{1}_{\{\tau<T\}}], \ \ \forall t\in[T-h,T].
\]
On the other hand, the reflected BSDE \eqref{my3} with data $(\eta,g,L)=(\xi,f^Y,h(\cdot,Y_\cdot,\mathbf{P}_{Y_\cdot}))$ admits a unique solution
\[(\widetilde{Y},{Z},{K})\in \mathcal{S}^{p}_{[T-h,T]}\times \mathcal{H}^{p,d}_{[T-h,T]}\times\mathcal{A}^p_{[T-h,T]}.\] It follows from Lemma \ref{my6} that $\widetilde{Y}=\Gamma(Y)=Y$, which implies that
$(Y,Z,K)$ is a solution to the {mean-field reflected} BSDE \eqref{my1}  on the time interval $[T-h,T].$\medskip

\noindent Let us now turn to the proof of uniqueness. Suppose $(Y^{\prime},Z^{\prime},K^{\prime})$ is also a solution to the {mean-field reflected} BSDE \eqref{my1}  on the time interval $[T-h,T]$. In the spirit of Lemma \ref{my6}, $Y^{\prime}$ is the fixed point of the map $\Gamma$, which indicates that $Y=Y^{\prime}$. Applying It\^o's formula to $\left|Y-Y^{\prime}\right|^2$ yields that $Z=Z^{\prime}$ and then $K=K^{\prime}$. This completes the proof.
\end{proof}

\bigskip\noindent Now we are in a position to complete the proof of the main result.\medskip

\begin{proof}[Proof of Theorem \ref{my1161}] The uniqueness of the global solution on $[0, T]$ is inherited from the uniqueness of the local solution on each small time interval. It suffices to prove the existence.\medskip

\noindent By Lemma \ref{my3.5}, there exists a constant $\delta>0$ depending only on $p$, $\lambda$, $\gamma_1$ and $\gamma_2$, such that the {mean-field reflected} BSDE \eqref{my1} admits a unique  solution
\[(Y^1,Z^1,K^1)\in \mathcal{S}^{p}_{[T-{\delta},T]}\times \mathcal{H}^{p,d}_{[T-{\delta},T]}\times\mathcal{A}^p_{[T-{\delta},T]}\] on the  time interval $[T-\delta,T]$. Next, taking  $T-{\delta}$ as the terminal time and
applying Lemma \ref{my3.5} again, the {mean-field reflected} BSDE \eqref{my1} admits a unique  solution \[(Y^2,Z^2,K^2)\in \mathcal{S}^{p}_{[T-2\delta,T-\delta]}\times \mathcal{H}^{p,d}_{[T-2\delta,T-\delta]}\times\mathcal{A}^p_{[T-2\delta,T-\delta]}\] on the
time interval $[T-2\delta,T-\delta]$.
Denote by
\begin{align*}
&{Y}_t=\sum\limits_{i=1}^2Y^i_t\mathbf{1}_{[T-i\delta,T-(i-1)\delta)}+Y^1_T\mathbf{1}_{\{T\}}, \  {Z}_t=\sum\limits_{i=1}^2Z^i_t\mathbf{1}_{[T-i\delta,T-(i-1)\delta)}+
Z^1_T\mathbf{1}_{\{T\}},
\\
&{K}_t=K^2_t\mathbf{1}_{[T-i\delta,T-(i-1)\delta)}+\left(K^2_{T-\delta}+K^1_t\right) \mathbf{1}_{[T-\delta,T]}.
\end{align*}
It is easy to check that
$( {Y}, {Z}, {K})\in\mathcal{S}^{p}_{[T-2\delta,T]}\times \mathcal{H}^{p,d}_{[T-2\delta,T]}\times\mathcal{A}^p_{[T-2\delta,T]}$ is a solution to the mean-field reflected BSDE \eqref{my1}. Repeating this procedure, we get a global solution $(Y,Z,K)\in\mathcal{S}^{p} \times \mathcal{H}^{p,d} \times \mathcal{A}^p$.
The proof of the theorem is complete.
\end{proof}

\section{Bounded terminal condition and obstacle}
In this section, we will use a linearization technique and a fixed point argument to investigate the quadratic case for the mean-field reflected BSDE \eqref{my1} with bounded terminal condition and obstacle. In comparison to the Lipschitz case, the BMO martingale theory plays a key role here.
\medskip

\noindent In what follows,  we make use of the following  conditions on the terminal condition $\xi$, the driver $f$ and the constraint $h$.
\begin{description}
\item[(H1)]  The terminal condition $\xi \in \mathcal{L}^{\infty}$ with $\xi\geq h(T,\xi,  \mathbf{P}_{\xi})$.
\item[(H2)] There exist three positive constants $\alpha,\beta$ and $\gamma$ such that for any $t\in[0,T]$, $y\in \mathbb{R}$, $v\in \mathcal{P}_1(\mathbb{R})$ and $z\in \mathbb{R}^{d}$
\begin{equation*}
|f(t,y,v,z)|\leq \alpha+\beta (|y|+W_1(v,\delta_0))+\frac{\gamma}{2}|z|^2.
\end{equation*}
\item[(H3)] The process $h(\cdot,y,v)\in \mathcal{S}^{\infty}$ is uniformly bounded with respect to $(t,\omega,y,v)$.
\item[(H4)] There exist two constants $\gamma_1,\gamma_2>0$  such that for any  $t\in[0,T]$, { $y_1,y_2\in \mathbb{R}$, $v_1,v_2\in \mathcal{P}_1(\mathbb{R})$}
\begin{equation*}
|h(t,y_1,v_1)-h(t,y_2,v_2)|\leq \gamma_1 |y_1-y_2|+\gamma_2{W_1(v_1,v_2)}.
\end{equation*}
\item[(H5)] There exists a constant $\kappa$ such that for each $t\in[0,T]$, $y_1,y_2\in \mathbb{R}$, $v_1,v_2\in \mathcal{P}_1(\mathbb{R})$ and $z_1,z_2\in \mathbb{R}^{d}$
\begin{equation*}
|f(t,y_1,v_1,z_1)-f(t,y_2,v_2,z_2)|\leq \beta \left( |y_1-y_2|+{W_1(v_1,v_2)} \right)+\kappa(1+|z_1|+|z_2))|z_1-z_2|.
\end{equation*}
\end{description}
\noindent We are now ready to state the main result of this section.
\begin{theorem}\label{my1162}
Assume that \emph{(H1)}-\emph{(H5)} are satisfied.  If $\gamma_1$ and $\gamma_2$ satisfy   \begin{align}\label{myq9502}
\gamma_1+\gamma_2<1,
\end{align}
then the quadratic {mean-field reflected} BSDE \eqref{my1}  admits a unique
 solution $(Y,Z,K)\in \mathcal{S}^{\infty}\times BMO\times\mathcal{A}$.
\end{theorem}

\medskip  \noindent In  order to prove Theorem \ref{my1162}, we need to analyze the quadratic solution map $\Gamma$.

\begin{lemma}\label{myq7925}
Assume that \emph{(H2)} and \emph{(H5)} are satisfied and that $\eta\in\mathcal{L}^{\infty}, U\in\mathcal{S}^{\infty}$. Then, the following quadratic   BSDE:
{\small \begin{align*}
y^{\tau}_t=\eta+\int^{\tau}_t f(s,U_s,\mathbf{P}_{U_s},z^{\tau}_s)ds-\int^{\tau}_t z^{\tau}_s dB_s,
\end{align*}
}
admits a unique solution $(y^{\tau},z^{\tau})\in\mathcal{S}^{\infty}\times BMO$.
\end{lemma}
\begin{proof}
The result is an immediate consequence of \cite[Theorem 7.3.3]{zhang2017}.
\end{proof}

\begin{lemma}\label{myq792}
	Assume that \emph{(H1)}-\emph{(H5)} are satisfied and  $U\in\mathcal{S}^{\infty}$ {with $U_T=\xi$}. Then, the following quadratic reflected BSDE:
\begin{align}\label{myq527}
\begin{cases}
&Y_t=\xi+\int_t^T f(s,{U_s},\mathbf{P}_{U_s},Z_s)ds-\int_t^T Z_sdB_s+K_T-K_t, \quad 0\le t\le T,\\
& Y_t\geq h(t,U_t, \mathbf{P}_{U_t}),\quad \forall t\in [0,T] \,\,\mbox{ and }\,\, \int_0^T (Y_t-h(t,U_t,\mathbf{P}_{U_t}))dK_t = 0,
\end{cases}
\end{align}
	admits a unique solution $(Y,Z,K)\in \mathcal{S}^{\infty}\times BMO \times\mathcal{A}$.
\end{lemma}
\begin{proof}
It follows from  \cite[Theorem 7.3.1]{zhang2017} that the comparison principle holds for the driver $(t,z)\mapsto f(t,U_t,P_{U_t},z)$ under Assumptions {(H2)} and {(H5)}. This, together with Lemma \ref{my6}, implies that $Y=\Gamma(U)$ for any solution $Y$ to the  quadratic  reflected BSDE \eqref{myq527}. Thus, it suffices to prove the existence. From Assumptions (H1)-(H4), it is easy to check that the driver $(t,z)\mapsto f(t,U_t,P_{U_t},z)$ and the obstacle $t\mapsto h(t,U_t, \mathbf{P}_{U_t})$ satisfy \cite[Conditions (H1)-(H3)]{KL}.
Applying \cite[Theorem 1]{KL}, the reflected BSDE \eqref{myq527} has a solution $(Y,Z,K)\in \mathcal{S}^{\infty}\times \mathcal{H}^{d,2} \times\mathcal{A}^2$. Using Lemma \ref{my7963}  {in Appendix}, we derive that $(Z,K)\in BMO\times\mathcal{A}$.   This completes the proof.
\end{proof}

\medskip
\noindent We are now ready to complete the proof of the main result of this section.\medskip

\begin{proof}[Proof of Theorem \ref{my1162}]
Let $U^i\in \mathcal{S}^{\infty}$, $i=1,2$.
It follows from Lemma \ref{myq792} that
\begin{align}\label{my1002}
\Gamma(U^i)_t:=\esssup_{\tau\in\mathcal{T}_{t}} y^{i,\tau}_t, \ \ \forall t\in[0,T],
\end{align}
in which $y^{i,\tau}_t$ is the solution to the  BSDE \eqref{myq528}. Following the proof of Lemma \ref{my3.5} step by step (noting that $(\beta_t) \in BMO$ in this case), we have for every $t\in[0,T]$
{ \begin{align*}
&y^{1,\tau}_t-y^{2,\tau}_t\\
&\!=\mathbf{E}^{\widetilde{\mathbf{P}}}_t\bigg[h (\tau,U^1_{\tau},(\mathbf{P}_{U^1_s})_{s=\tau}) \mathbf{1}_{\{\tau<T\}}-h (\tau,U^2_{\tau},(\mathbf{P}_{U^2_s})_{s=\tau}) \mathbf{1}_{\{\tau<T\}} +\int^{\tau}_t \left(f^{U^1}(s,z^{2,\tau}_s)-f^{U^2}(s,z^{2,\tau}_s)\right)ds\bigg],
\end{align*}}
which together with assumptions (H4) and (H5) implies that for any $t\in[T-h,T]$,
 \begin{align*}
|y^{1,\tau}_t-y^{2,\tau}_t|\leq (\gamma_1+\beta h)\|U^1-U^2\|_{\mathcal{S}^{\infty}_{[T-h,T]}}+(\gamma_2+\beta h)\sup\limits_{s\in[T-h,T]}\mathbf{E}[|U_s^1-U^2_s|].
\end{align*}
The above inequality combined with \eqref{my1002} implies the following,
 \begin{align*}\begin{split}
\|\Gamma(U^1)-\Gamma(U^2)\|_{\mathcal{S}^{\infty}_{[T-h,T]}}
\leq (\gamma_1+\gamma_2+2\beta h)\|U^1-U^2\|_{\mathcal{S}^{\infty}_{[T-h,T]}}.
\end{split}
\end{align*}
Under assumption \eqref{myq9502}, we can then find a small enough constant $h$ depending only on $\beta,\gamma_1$ and $\gamma_2$ such that $\gamma_1+\gamma_2+2\beta h< 1.$ It is now obvious that $\Gamma$ defines a contraction map on the time interval $[T-h,T]$.
Finally, proceeding exactly as in Theorem \ref{my1161} and Lemma \ref{my3.5}, we complete the proof.
\end{proof}

\begin{remark}{\upshape
When the terminal condition is unbounded, the process $(\beta_t)$ may be unbounded in the BMO space, so that $\widetilde{\mathbf{P}}$ is not well-defined. Thus, the conventional fixed point argument fails to work in the unbounded terminal condition case.
}
\end{remark}

\section{Unbounded terminal condition and obstacle}
In this section, we will use the $\theta$-method to  deal with quadratic mean-field reflected BSDEs taking the form \eqref{my1} with unbounded terminal condition and obstacle. For this purpose, we need to assume the driver is concave or convex with respect to the second unknown $z$. In what follows, we make use of the following  conditions on the terminal condition $\xi$, the driver $f$ and the constraint $h$.
\begin{description}
\item[(H1')]  The terminal condition $\xi \in \mathbb{L}$ with $\xi\geq h(T,\xi,  \mathbf{P}_{\xi})$.
\item[(H3')] For any {$y\in\mathbb{R}$, $v\in \mathcal{P}_1(\mathbb{R})$}, the process $h(t,y,v)$ belongs to $\mathbb{S}$ .
\item[(H5')] For each $t\in[0,T]$, $y_1,y_2\in \mathbb{R}$, $v_1,v_2\in \mathcal{P}_1(\mathbb{R})$ and $z\in \mathbb{R}^{d}$
\begin{equation*}
|f(t,y_1,v_1,z)-f(t,y_2,v_2,z)|\leq \beta \left( |y_1-y_2|+{W_1(v_1,v_2)} \right).
\end{equation*}
\item[(H6)] For each $(t,\omega,y,v)\in[0,T]\times\Omega\times\mathbb{R}\times\mathcal{P}_1(\mathbb{R})$, $f(t,y,v,\cdot)$ is  concave or convex.
\end{description}
We are now ready to state the main result of this section.
\begin{theorem}\label{my116}
Assume that \emph{(H1')}, \emph{(H2)}, \emph{(H3')}, \emph{(H4)}, \emph{(H5')} and \emph{(H6)} hold.  If $\gamma_1$ and $\gamma_2$ satisfy   \begin{align}\label{myq950}
4(\gamma_1+\gamma_2)<1,
\end{align}
then the quadratic {mean-field reflected} BSDE \eqref{my1}  admits a unique
 solution $(Y,Z,K)\in \mathbb{S}\times \mathcal{H}^{d}\times\mathcal{A}$.
\end{theorem}

\noindent In order to prove Theorem \ref{my116}, we need to recall some technical results on the representation of solutions of quadratic BSDEs. First, we introduce some general conditions on the generator.
\begin{description}
\item[(H7)] There exists a positive progressively measurable   process $(\alpha_{t})_{0\leq t\leq T}$ with   $\int^T_0\alpha_tdt\in\mathbb{L}$ such that for each  $(t,\omega, z)\in[0,T]\times\Omega\times\mathbb{R}\times\mathbb{R}^d$, $|g(t,y,z)|\leq \alpha_t+\beta|y|+\frac{\gamma}{2}|z|^2.$
\item[(H8)] For each  $(t,\omega, z)\in[0,T]\times\Omega\times\mathbb{R}^d$, $g(t,y,z)\leq \alpha_t+\beta|y|+\frac{\gamma}{2}|z|^2.$
\end{description}
\begin{remark}\label{my900}{\upshape
Suppose that $g$ satisfies assumptions  (H5'), (H6) and (H7). It follows from  \cite[Corollary 6]{BH2008} that the quadratic BSDE \eqref{myq2} with $\eta\in\mathbb{L}$ admits a unique solution $(y^{\tau},z^{\tau})\in \mathbb{S}\times\mathcal{H}^{d}.$  In particular, the comparison principle also holds by \cite[Proposition 5.1]{BY} or  \cite[Theorem 5]{BH2008} (up to a slight modification). Under assumptions  (H2), (H5') and (H6), the quadratic reflected BSDE \eqref{my3} with $L\in\mathbb{S}$ admits a unique solution $(Y,Z,K)\in \mathbb{S}\times\mathcal{H}^{d}\times\mathcal{A}$ by \cite[Theorems 3.2 and 4.1]{BY}. }
\end{remark}

\noindent The following result plays a key role in our subsequent calculus, and can be derived from \cite[Proposition 1]{FHT1}.
\begin{lemma}\label{my7}
Assume that $(y^{\tau},z^{\tau})\in \mathcal{S}^2\times\mathcal{H}^{2,d}$  is a solution to \eqref{myq2}. Suppose that there is a  constant $p\geq 1$ such that
\begin{align*}
\mathbf{E}\bigg[\exp\bigg\{2p\gamma e^{\beta T}\sup\limits_{t\in[0,T]}|y^{\tau}_t|+2p\gamma\int^T_0\alpha_t e^{\beta t}dt\bigg\}\bigg]<\infty.
\end{align*}
Then, we have
	\begin{description}
		\item[(i)] Let Assumption (H7) hold. Then, for each $t\in[0,T]$, Then for each $t\in[0,T]$ and $p\geq 1$,
\begin{align*}\exp\left\{p\gamma|y^{\tau}_t|\right\}\leq   \mathbf{E}_t\bigg[\exp\bigg\{p\gamma e^{\beta (T-t)}|\eta|+p\gamma\int^T_t\alpha_s e^{\beta (s-t)}ds\bigg\}	 \bigg].
		\end{align*}
				\item[(ii)] Let Assumption (H8) hold. Then, for each $t\in[0,T]$,
 \begin{align*}
		\exp\left\{{p\gamma}(y^{\tau}_t)^+\right\}\leq   \mathbf{E}_t\bigg[\exp\bigg\{p\gamma e^{\beta (T-t)}\eta^++p\gamma\int^T_t\alpha_s e^{\beta (s-t)}ds\bigg\}
		\bigg].
		\end{align*}	
	\end{description}
\end{lemma}

\medskip
\noindent We are now ready to prove Theorem \ref{my116}. Indeed, we will make use of the $\theta$-method to prove existence and uniqueness of the solution of the quadratic mean-field reflected BSDE \eqref{my1}. \medskip

\begin{lemma}\label{myq7900}
	Assume that all the conditions of Theorem \ref{my116} hold. Then, the  quadratic mean-field reflected BSDE \eqref{my1} has at most one solution $(Y,Z,K)\in\mathbb{S}\times \mathcal{H}^{d}\times\mathcal{A}$.
\end{lemma}
\begin{proof}
\noindent For $i=1,2$,
let $({Y}^i,{Z}^i,{K}^i)$ be a  $ \mathbb{S}\times\mathcal{H}^d\times\mathcal{A}$-solution to the quadratic mean-field reflected BSDE~\eqref{my1}. From Lemma \ref{my6} and Remark \ref{my900}, we have
\begin{align*}
Y^{i}_t:=\esssup_{\tau\in\mathcal{T}_{t}} y^{i,\tau}_t, \ \ \forall t\in[0,T],
\end{align*}
in which $y^{i,\tau}_t$ is the solution of the following quadratic BSDE
{\small \begin{align*}
y^{i,\tau}_t=\xi\mathbf{1}_{\{\tau=T\}}+h (\tau,Y^{i}_{\tau},{(\mathbf{P}_{Y^{i}_s})_{s=\tau}}) \mathbf{1}_{\{\tau<T\}}+\int^{\tau}_t f(s,Y^i_s,\mathbf{P}_{Y^{i}_s},z^{i,\tau}_s)ds-\int^{\tau}_t z^{i,\tau}_s dB_s.
\end{align*}
}

\noindent Assume without loss of generality that $f(t,y,v,\cdot)$ is concave (see Remark \ref{myrk11}), for each $\theta\in (0,1)$, denote by
\[ \delta_{\theta}\ell=\frac{\theta \ell^1-\ell^2}{1-\theta}, \ \delta_{\theta}\widetilde{\ell}=\frac{\theta \ell^2-\ell^1}{1-\theta} ~~\text{and} ~~ \delta_{\theta}\overline{Y}:= |\delta_{\theta}Y|+|\delta_{\theta}\widetilde{Y}|\]
for $\ell=Y,y^{\tau}$ and $z^{\tau}$. Then, the pair of processes $(\delta_{\theta}y^{\tau},\delta_{\theta}z^{\tau})$
satisfies the following BSDE:
 \begin{align}\label{myq123}
\begin{split}
\delta_{\theta}y^{\tau}_t=&\delta_{\theta}\eta+\int^\tau_t\left(\delta_{\theta}f(s,\delta_{\theta}z^{\tau}_s)+\delta_{\theta}f_0(s)\right)ds-\int^\tau_t\delta_{\theta}z^{\tau}_sdB_s,
\end{split}
\end{align}
where the terminal condition and generator are given by
\begin{align*}
&\delta_{\theta}\eta=-\xi\mathbf{1}_{\{\tau=T\}}+\frac{\theta h (\tau,Y^{1}_{\tau},{(\mathbf{P}_{Y^{1}_s})_{s=\tau}})- h (\tau,Y^{2}_{\tau},{(\mathbf{P}_{Y^{2}_s})_{s=\tau}})}{1-\theta} \mathbf{1}_{\{\tau<T\}},\\
& \delta_{\theta}f_0(t)=
\frac{1}{1-\theta}\left(f(t,Y^1_{t},\mathbf{P}_{Y^{1}_t}, z^{2,\tau}_t)-f(t,Y^2_{t},\mathbf{P}_{Y^{2}_t}, z^{2,\tau}_t)\right),\\
&
\delta_{\theta}f(t,z)=\frac{1}{1-\theta}\bigg(
\theta f(t,Y^1_t,\mathbf{P}_{Y^{1}_t}, z^{1,\tau}_t)- f(t,Y^1_t,\mathbf{P}_{Y^{1}_t}, -(1-\theta)z+\theta z^{1,\tau}_t)\bigg).
\end{align*}

\noindent Recalling assumptions (H2), (H4), (H5') and (H6), we have that
\begin{align*}
&\delta_{\theta}\eta\leq |\xi|+|h(\tau,0,\delta_0)|+\gamma_1( 2|Y^{1}_{\tau}|+|\delta_{\theta}Y_\tau|) +\gamma_2(2\mathbf{E}[|Y^{1}_{s}|]_{s=\tau}+\mathbf{E}[|\delta_{\theta}Y_s|]_{s=\tau})),\\
 &\delta_{\theta}f_0(t)\leq \beta (|Y^1_t|+|\delta_{\theta}Y_t|+\mathbf{E}[|Y^1_t|]+\mathbf{E}[|\delta_{\theta}Y_t|]),\\
&\delta_{\theta}f(t,z)
 \leq -f(t,Y^1_t,\mathbf{P}_{Y^{1}_t}, -z)
 \leq \alpha+\beta (|Y^1_t|+\mathbf{E}[|Y^{1}_{t}|])+\frac{\gamma}{2}|z|^2.
\end{align*}
Set $ C_1:=\sup\limits_{i\in\{1,2\}}\mathbf{E}\big[\sup\limits_{s\in[0,T]}|Y^{i}_{s}|\big]$ and
{ \begin{align*}
&\chi= \alpha T+\sup\limits_{s\in[0,T]}|h(s,0,\delta_0)|+2(\gamma_1+\beta T)\left(\sup\limits_{s\in[0,T]}|Y^{1}_{s}|+\sup\limits_{s\in[0,T]}|Y^{2}_{s}|\right)+2(\gamma_2+\beta T)C_1,\\
&\widetilde{\chi}= \alpha T+\sup\limits_{s\in[0,T]}|h(s,0,\delta_0)|+2(1+\gamma_1+\beta T)\left(\sup\limits_{s\in[0,T]}|Y^{1}_{s}|+\sup\limits_{s\in[0,T]}|Y^{2}_{s}|\right)+2(\gamma_2+\beta T)C_1.\end{align*}
}

\noindent Using assertion (ii) of Lemma \ref{my7} to  \eqref{myq123}, we derive that  for any $p\geq 1$
{ \begin{align*}	
 \begin{split}
 &\exp\left\{{p\gamma}\big(\delta_{\theta}y^{\tau}_t\big)^+\right\}\\
 &\leq   \mathbf{E}_t\bigg[\exp\bigg\{p\gamma \bigg(|\xi|+\chi+(\gamma_1+\beta (T-t))\sup\limits_{s\in[t,T]}|\delta_{\theta}Y_{s}|+(\gamma_2+\beta (T-t))\sup\limits_{s\in[t,T]}\mathbf{E}[|\delta_{\theta}Y_{s}|] \bigg)\bigg\}	 \bigg],
 \end{split}
	\end{align*}
}
which indicates that
{ \begin{align}	\label{myq2131}
 \begin{split}
 &\exp\left\{{p\gamma}\left(\delta_{\theta}Y_t\right)^+\right\}\leq \esssup_{\tau\in\mathcal{T}_{t}}\exp\left\{{p\gamma}\left(\delta_{\theta}y^{\tau}_t\right)^+\right\}\\
 &\leq   \mathbf{E}_t\bigg[\exp\bigg\{p\gamma \bigg(|\xi|+\chi+(\gamma_1+\beta (T-t))\sup\limits_{s\in[t,T]}|\delta_{\theta}Y_{s}|+(\gamma_2+\beta (T-t))\sup\limits_{s\in[t,T]}\mathbf{E}[|\delta_{\theta}Y_{s}|] \bigg)\bigg\}	 \bigg].
 \end{split}
	\end{align}
}

\noindent Using a similar method, we derive that
{ \begin{align}	\label{myq2031}
 \begin{split}
 &\exp\left\{{p\gamma}\left(\delta_{\theta}\widetilde{Y}_t\right)^+\right\}\\
 &\leq   \mathbf{E}_t\bigg[\exp\bigg\{p\gamma \bigg(|\xi|+\chi+(\gamma_1+\beta (T-t))\sup\limits_{s\in[t,T]}|\delta_{\theta}\widetilde{Y}_{s}|+(\gamma_2+\beta (T-t))\sup\limits_{s\in[t,T]}\mathbf{E}[|\delta_{\theta}\widetilde{Y}_{s}|] \bigg)\bigg\}	\bigg].
 \end{split}
	\end{align}
}

\noindent In view of the fact that
\begin{align*}
\left(\delta_{\theta}{Y}\right)^-
\leq \left(\delta_{\theta}\widetilde{Y}\right)^++2|Y^{2}| \ \text{and}\ \left(\delta_{\theta}\widetilde{Y}\right)^-
\leq \left(\delta_{\theta}{Y}\right)^++2|Y^{1}|,
\end{align*}
and recalling \eqref{myq2131} and \eqref{myq2031}, we have
{  \begin{align*}
	\begin{split}
	&\exp\left\{p\gamma \left|\delta_{\theta}{Y}_t\right|\right\}\vee \exp\left\{p\gamma \left|\delta_{\theta}\widetilde{Y}_t\right|\right\} \leq
	\exp\left\{{p\gamma}\left(\left( \delta_{\theta}{Y}_t\right)^++\left(\delta_{\theta}\widetilde{Y}_t\right)^++2|Y^{1}_t|+2|Y^{2}_t|\right)\right\}\\
	&\leq      \mathbf{E}_t\bigg[\exp\bigg\{p\gamma \bigg(|\xi|+\widetilde{\chi}+(\gamma_1+\beta (T-t))\sup\limits_{s\in[t,T]}\delta_{\theta}\overline{Y}_{s}+(\gamma_2+\beta (T-t))\sup\limits_{s\in[t,T]}\mathbf{E}[\delta_{\theta}\overline{Y}_{s}] \bigg)\bigg\}	\bigg]^2.
	\end{split}
	\end{align*}
}
Applying {Doob's maximal inequality} and
H\"{o}lder's inequality, we get  that for each  $p\geq 1$ and $t\in[0,T]$
{\small  \begin{align}\label{myq2231}
	\begin{split}
	&\mathbf{{E}}\bigg[\exp\big\{p\gamma \sup\limits_{s\in[t,T]}\delta_{\theta}\overline{Y}_s\big\}\bigg]\leq \mathbf{{E}}\bigg[\exp\big\{p\gamma \sup\limits_{s\in[t,T]}|\delta_{\theta}{Y}_s|\big\}\exp\big\{p\gamma \sup\limits_{s\in[t,T]}|\delta_{\theta}\widetilde{Y}_s|\big\}\bigg]
	\\ &\leq  4  \mathbf{E}\bigg[\exp\bigg\{4p\gamma \bigg(|\xi|+\widetilde{\chi} +(\gamma_1+\beta (T-t))\sup\limits_{s\in[t,T]}\delta_{\theta}\overline{Y}_{s} +(\gamma_2+\beta (T-t))\sup\limits_{s\in[t,T]}\mathbf{E}[\delta_{\theta}\overline{Y}_{s}] \bigg)\bigg\}	\bigg]
	\\ &\leq 4  \mathbf{E}\bigg[\exp\bigg\{4p\gamma \bigg(|\xi|+\widetilde{\chi} +(\gamma_1+\beta (T-t))\sup\limits_{s\in[t,T]}\delta_{\theta}\overline{Y}_{s}\bigg)\bigg\}	\bigg] \mathbf{E}\bigg[\exp\bigg\{4p\gamma (\gamma_2+\beta (T-t))\sup\limits_{s\in[t,T]}\delta_{\theta}\overline{Y}_{s} \bigg\}	\bigg],
	\end{split}
	\end{align}
}where we used  Jensen's inequality  in the last inequality.

\medskip \noindent  Under assumption \eqref{myq950}, there exist two  constants $h\in (0,T]$ and $\nu>1$  depending only on $\beta,\gamma_1$ and $\gamma_2$ such that \begin{align*}
\text{ $ 4(\gamma_1+\gamma_2+2\beta h) < 1 $  and $ 4\nu (\gamma_1+\beta h)<1 $.}
\end{align*}
In the spirit of H\"{o}lder's inequality, we derive that for any  $p\geq 1$
{ \small \begin{align*}	
	\begin{split}
	&\mathbf{{E}}\bigg[\exp\big\{p\gamma \sup\limits_{s\in[T-h,T]}\delta_{\theta}\overline{Y}_s\big\}\bigg]\\&\leq 4\mathbf{{E}}\bigg[\exp\bigg\{\frac{4\nu p\gamma}{\nu-1}(|\xi|+\widetilde{\chi})\bigg\}		 \bigg]^{\frac{\nu-1}{\nu}}  \mathbf{E}\bigg[\exp\bigg\{4\nu p\gamma (\gamma_1+\beta h)\sup\limits_{s\in[t,T]}\delta_{\theta}\overline{Y}_{s}\bigg\}	\bigg]^{\frac{1}{\nu}} \mathbf{E}\bigg[\exp\bigg\{p\gamma \sup\limits_{s\in[t,T]}\delta_{\theta}\overline{Y}_{s} \bigg\}	\bigg]^{4(\gamma_2+\beta h)}\\
	&\leq 4\mathbf{{E}}\bigg[\exp\bigg\{\frac{4\nu p\gamma}{\nu-1}(|\xi|+\widetilde{\chi})\bigg\}		\bigg]^{\frac{\nu-1}{\nu}}\mathbf{{E}}\bigg[\exp\bigg\{p\gamma\sup\limits_{s\in[T-h,T]}\delta_{\theta}\overline{Y}_{s}\bigg\}		 \bigg]^{4(\gamma_1+\gamma_2+2\beta h)},
	\end{split}
	\end{align*}
}which together with the fact that $4(\gamma_1+\gamma_2+2\beta h)<1$ implies that for any $p\geq 1$ and $\theta\in(0,1)$
{ \[
\mathbf{{E}}\bigg[\exp\big\{p\gamma \sup\limits_{s\in[T-h,T]}\delta_{\theta}\overline{Y}_s\big\}\bigg]\leq 4^{\frac{1}{1-4(\gamma_1+\gamma_2+2\beta h)}}\mathbf{{E}}\bigg[\exp\bigg\{\frac{4\nu p\gamma}{\nu-1}(|\xi|+\widetilde{\chi})\bigg\}		\bigg]^{\frac{\nu-1}{\nu(1-4(\gamma_1+\gamma_2+2\beta h))}}<\infty.
\]
}Note that
${Y}^{1}-{Y}^{2}= (1-\theta)(\delta_{\theta}{Y}+Y^{1}). $ It follows that
\begin{align*}
\mathbf{E}\bigg[\sup\limits_{t\in[T-h,T]}
\big|{Y}^{1}_t-{Y}^{2}_t\big|\bigg]\leq (1-\theta)\bigg(\frac{1}{\gamma}\mathbf{{E}}\bigg[4^{\frac{\nu}{\nu-1}}\exp\bigg\{\frac{4\nu \gamma}{\nu-1}(|\xi|+\widetilde{\chi})\bigg\}		 \bigg]^{\frac{\nu-1}{\nu(1-4(\gamma_1+\gamma_2+2\beta h))}}+\mathbf{{E}}\bigg[\sup\limits_{t\in[0,T]}\big|{Y}^{1}_t\big|\bigg]\bigg).
\end{align*}
Letting $\theta\rightarrow 1$ yields that $Y^1=Y^2$ and then $(Z^1,K^1)=(Z^2,K^2)$ on $[T-h,T]$. Repeating iteratively this procedure a finite number of times, we get the uniqueness on the given interval $[0,T]$. The proof is complete.
\end{proof}

\begin{remark}\label{myrk11}
    \noindent  {\upshape In the convex case, one should use $\ell^1-\theta \ell^{2}$ and   $\ell^{2}-\theta \ell^{1}$ instead of
 $\theta \ell^{1}- \ell^{2}$ and  $\theta \ell^{2}- \ell^{1}$  in the definition of $\delta_{\theta}\ell$ and $\delta_{\theta}\widetilde{\ell}$, respectively. Then the terminal condition and generator of  BSDE \eqref{myq123} satisfies
 \begin{align*}
&\delta_{\theta}\eta\leq |\xi|+|h(\tau,0,\delta_0)|+\gamma_1( 2|Y^{2}_{\tau}|+|\delta_{\theta}Y_\tau|) +\gamma_2(2\mathbf{E}[|Y^{2}_{s}|]_{s=\tau}+\mathbf{E}[|\delta_{\theta}Y_s|]_{s=\tau})),\\
 &\delta_{\theta}f_0(t)\leq \beta (|Y^2_t|+|\delta_{\theta}Y_t|+\mathbf{E}[|Y^2_t|]+\mathbf{E}[|\delta_{\theta}Y_t|]),\\
&\delta_{\theta}f(t,z)
 \leq f(t,Y^2_t,\mathbf{P}_{Y^{2}_t}, z)
 \leq \alpha+\beta (|Y^2_t|+\mathbf{E}[|Y^{2}_{t}|])+\frac{\gamma}{2}|z|^2.
\end{align*}
By a similar analysis, one can check that \eqref{myq2131}, \eqref{myq2031} and \eqref{myq2231} still hold.
}
\end{remark}

\begin{remark}{\upshape
Note that we do not obtain directly a uniform estimate for $\left(\delta_{\theta}Y_t\right)^+$  in \eqref{myq2131}, which involves the term $|\delta_{\theta}Y_t|$. Otherwise, the condition \eqref{myq950} could reduce to the condition \eqref{myq9502}.
}
\end{remark}

\noindent Let us now turn to the proof of existence.
\begin{lemma}\label{myq79}
	Assume that all the conditions of Theorem \ref{my116} hold and $U\in\mathbb{S}$ with $U_T=\xi$. Then, the following quadratic reflected BSDE:
\begin{align*}
\begin{cases}
&Y_t=\xi+\int_t^T f(s,Y_s,\mathbf{P}_{U_s},Z_s)ds-\int_t^T Z_sdB_s+K_T-K_t, \quad 0\le t\le T,\\
& Y_t\geq h(t,U_t, \mathbf{P}_{U_t}),\quad \forall t\in [0,T] \,\,\mbox{ and }\,\, \int_0^T (Y_t-h(t,U_t,\mathbf{P}_{U_t}))dK_t = 0,
\end{cases}
\end{align*}
	admits a unique solution $(Y,Z,K)\in \mathbb{S}\times\mathcal{H}^d\times\mathcal{A}$.
\end{lemma}
\begin{proof}
	From Assumptions (H2), (H4) and (H6), we have
{\small	\begin{align}\label{myq82}
	|f(t,y,P_{U_t},z)|\leq \alpha+\beta(y+\mathbf{E}[|U_{t}|])+\frac{\gamma}{2}|z|^2 ,~~ |h(t,U_{t},\mathbf{P}_{U_t})| \leq |h(t,0,\delta_0)|+\gamma_1|U_{t}|+\gamma_2\mathbf{E}[|U_{t}|].
	\end{align}}
\noindent It follows from \cite[Proposition 5.1]{BY} or \cite[Theorem 5]{BH2008} that the driver $(t,y,z)\mapsto f(t,y,P_{U_t},z)$ satisfies the comparison principle. Then, the above  quadratic reflected BSDE  has at most one $\mathbb{S}\times\mathcal{H}^d\times\mathcal{A}$-solution by Lemma \ref{my6}.	In particular, $(t,y,z)\mapsto f(t,y,P_{U_t},z)$ satisfies \cite[Conditions (H1)]{BY} and $h(\cdot,U_{\cdot},\mathbf{P}_{U_\cdot})\in\mathbb{S}$.
Consequently, applying \cite[Theorem 3.2]{BY}, we get the desired result.
\end{proof}
\begin{remark}{\upshape
For a process $U\in\mathbb{S}$, the driver $(t,z)\mapsto f(t,U_t,P_{U_t},z)$ may not satisfy \cite[Condition (H1))]{BY}. We use the driver $(t,y,z)\mapsto f(t,y,P_{U_t},z)$ instead in Lemma \ref{myq79}.
}
\end{remark}

\begin{lemma}\label{myq520}
Assume that $\xi\in\mathbb{L}$. Then, the process $(\mathbf{E}_t[\xi])_{0\leq t\leq T} \in \mathbb{S}.$
\end{lemma}
\begin{proof}
 Using  Jensen's inequality yields that
\begin{align*}\exp\big\{|\mathbf{E}_t[\xi]|\big\}\leq   \mathbf{E}_t\big[\exp\big\{|\xi|\big\}	 \big],
		\end{align*}
		which together with  {Doob's maximal inequality} indicates that
		\begin{align*}\mathbf{E}\bigg[\exp\big\{p\sup\limits_{t\in[0,T]}|\mathbf{E}_t[\xi]|\big\}\bigg]\leq\bigg(\frac{p}{p-1}\bigg)^p   \mathbf{E}\big[\exp\big\{p|\xi|\big\}	 \big]<\infty, \ \forall p>1.
		\end{align*}
The proof is complete.
\end{proof}
\medskip

\noindent
Then, based on Lemmas \ref{myq79} and \ref{myq520}, we could define recursively a sequence of stochastic processes $(Y^{(m)})_{m=1}^{\infty}$ through  the following quadratic reflected BSDE:
{\small \begin{align}\label{myq78}
\begin{cases}
&Y_t^{(m)}=\xi+\int_t^T f(s,Y^{(m)}_s,\mathbf{P}_{Y^{(m-1)}_s},Z^{(m)}_s)ds-\int_t^T Z^{(m)}_sdB_s+K^{(m)}_T-K^{(m)}_t, \quad 0\le t\le T,\\
& Y^{(m)}_t\geq h(t,Y^{(m-1)}_t, \mathbf{P}_{Y^{(m-1)}_t}),\quad \forall t\in [0,T] \,\,\mbox{ and }\,\, \int_0^T (Y^{(m)}_t-h(t,Y^{(m-1)}_t,\mathbf{P}_{Y^{(m-1)}_t}))dK^{(m)}_t = 0,
\end{cases}
\end{align}}

\noindent where $Y^{(0)}_t=\mathbf{E}_t[\xi]$ for $t\in[0,T]$.
In particular, we have that $(Y^{(m)},Z^{(m)},K^{(m)})\in \mathbb{S}\times\mathcal{H}^d\times\mathcal{A}$. Next, we use  a $\theta$-method to prove that the limit of $Y^{(m)}$ is a desired solution. The following uniform estimates {are} crucial for our main result.

\begin{lemma}\label{myq7901}
	Assume that the conditions of Theorem \ref{my116} are fulfilled. Then, for any $p\geq 1$, we have
	\begin{align*}
	\begin{split}
\sup\limits_{m\geq 0}\mathbf{E}\left[\exp\bigg\{p\gamma\sup\limits_{s\in[0,T]}|Y^{(m)}_s|\bigg\}+\bigg(\int^T_0|Z^{(m)}_t|^2dt\bigg)^p+|K^{(m)}_T|^p \right]<\infty.
	\end{split}
	\end{align*}
\end{lemma}
\begin{proof}
The proof will be given in Appendix.
\end{proof}

\begin{lemma}\label{myq7902}
	Assume that all the conditions of Theorem \ref{my116} hold. Then, for any $p\geq 1$, we have
	\begin{align*}
	\begin{split}
\Pi(p):=\sup\limits_{\theta\in(0,1)}\lim_{m\rightarrow \infty}\sup\limits_{q\geq 1}\mathbf{{E}}\bigg[\exp\left\{p\gamma \sup\limits_{s\in[0,T]}\delta_{\theta}\overline{Y}^{(m,q)}_s\right\}\bigg]<\infty,
	\end{split}
	\end{align*}
	where we use the following notations
	{ \[
\delta_{\theta}Y^{(m,q)}=\frac{\theta Y^{(m+q)}-Y^{m}}{1-\theta}, \ \delta_{\theta}\widetilde{Y}^{(m,q)}=\frac{\theta Y^{(m)}-Y^{(m+q)}}{1-\theta}~ \text{and} ~
\delta_{\theta}\overline{Y}:=|\delta_{\theta} Y^{(m,q)}|+|\delta_{\theta}\widetilde{Y}^{(m,q)}|.
\]
}
\end{lemma}
\begin{proof}
The proof will be given in Appendix.
\end{proof}
\medskip

 \noindent We are now in a position to complete the proof of the main result.\medskip

\begin{proof}[Proof of Theorem \ref{my116}] It is enough to prove the existence, the uniqueness was dealt with in Lemma \ref{myq7900}. Note that for any integer $p\geq 1$ and $\theta\in (0,1)$,
\begin{align*}
\limsup_{m\rightarrow \infty}\sup\limits_{q\geq 1}\mathbf{E}\big[\sup\limits_{t\in[0,T]}
\big|{Y}^{(m+q)}_t-{Y}^{(m)}_t\big|^p\big]\leq 2^{p-1}(1-\theta)^p\bigg(\frac{\Pi(1)p!}{\gamma^{p}}+\sup\limits_{m\geq 1}\mathbf{{E}}\big[\sup\limits_{t\in[0,T]}\big|{Y}^{(m)}_t\big|^p\big]\bigg).
\end{align*}
Sending $\theta\rightarrow 1$ and recalling Lemmas \ref{myq7901} and \ref{myq7902}, we could find a continuous process $Y\in\mathbb{S}$ such that
\begin{align}\label{myq37}
\lim_{m\rightarrow \infty} \mathbf{E}\bigg[\sup\limits_{t\in[0,T]}
\big|{Y}^{(m)}_t-{Y}_t\big|^p\bigg]=0,\ \forall p\geq 1.
\end{align}
Applying It\^o's formula to $\big|{Y}^{(m+q)}_t-{Y}^{(m)}_t\big|^2$  and by a standard calculus, we have
\begin{align*}
\begin{split}
&\mathbf{E}\bigg[\int^T_0\big|Z^{(m+q)}_t-Z^{(m)}_t\big|^2dt \bigg]\leq \mathbf{E}\bigg[\sup\limits_{t\in[0,T]}
\big|{Y}^{(m+q)}_t-{Y}^{(m)}_t\big|^2+
\sup\limits_{t\in[0,T]}
\big|{Y}^{(m+q)}_t-{Y}^{(m)}_t\big|\Delta^{(m,q)}\bigg]\\
&\leq \mathbf{E}\bigg[\sup\limits_{t\in[0,T]}
\big|{Y}^{(m+q)}_t-{Y}^{(m)}_t\big|^2\bigg]+
\mathbf{E}\left[|\Delta^{(m,q)}|^2\right]^{\frac{1}{2}}\mathbf{E}\bigg[\sup\limits_{t\in[0,T]}
\big|{Y}^{(m+q)}_t-{Y}^{(m)}_t\big|^2\bigg]^{\frac{1}{2}}
\end{split}
\end{align*}
with
\[
\Delta^{(m,q)}:=\int^T_0\big|f(t,Y^{(m+q)}_t,\mathbf{P}_{Y^{(m+q-1)}_t},Z^{(m+q)}_t)- f(t,Y^{(m)}_t,\mathbf{P}_{Y^{(m-1)}_t},Z^{(m)}_t)\big|dt+|K_T^{(m+q)}|+|K_T^{(m)}|.
\]
which together with   Lemma \ref{myq7901}, \eqref{myq37} and dominated convergence theorem indicates that  there exists a process $Z\in \mathcal{H}^d$ so that
\begin{align}\label{myq51}
\lim_{m\rightarrow \infty} \mathbf{E}\bigg[\bigg(\int^T_0\big|Z^{(m)}_t-Z_t\big|^2dt\bigg)^p\bigg]=0, \ \forall p\geq 1.
\end{align}
Set
\[
K_t=Y_t-Y_0+\int_0^tf(s,Y_s,\mathbf{P}_{Y_s},Z_s)\, ds-\int_0^tZ_s\, dB_s.
\]
Applying dominated convergence theorem again yields that for each $p\geq 1$,
\begin{align*}
\lim_{m\rightarrow \infty}\mathbf{E}\left[\bigg(\int^T_0|f(t, Y_t^{(m)},\mathbf{P}_{Y_t^{(m-1)}}, Z_t^{(m)})-f(t, Y_t, \mathbf{P}_{Y_t}, Z_t)|dt\bigg)^p\right]=0,
\end{align*}
which implies that  $\mathbf{E}\big[\sup\limits_{t\in[0,T]}\big|K_t-K_t^{(m)}\big|^p\big]\rightarrow 0$ as $m\rightarrow \infty$ for each $p\geq 1$ and that $K$ is a non-decreasing process. Note that
\[ \lim\limits_{m\rightarrow\infty}\mathbf{E}\bigg[\sup\limits_{t\in[0,T]}\big|h(t,Y^{(m-1)}_t, \mathbf{P}_{Y^{(m-1)}_t})-h(t,Y_t, \mathbf{P}_{Y_t})\big|\bigg]\leq (\gamma_1+\gamma_2) \lim_{m\rightarrow \infty}\mathbf{{E}}\big[\sup\limits_{t\in[0,T]} \big|{Y}^{(m-1)}_t-{Y}_t\big|\big]=0.\]
Then it is obvious that $Y_t\geq h(t,Y_t, \mathbf{P}_{Y_t}).$
Moreover, recalling \cite[Lemma 13]{BH}, we have\[
\int_0^T (Y_t-h(t,Y_t,\mathbf{P}_{Y_t}))dK_t=\lim\limits_{m\rightarrow\infty} \int_0^T (Y^{(m)}_t-h(t,Y^{(m-1)}_t, \mathbf{P}_{Y^{(m-1)}_t}))dK^{(m)}_t=0,
\]
 which implies that $(Y,Z,K)\in\mathbb{S}\times\mathcal{H}^d\times\mathcal{A}$ is a solution to quadratic mean-field reflected \eqref{my1}. The proof is complete.
\end{proof}

\appendix
\renewcommand\thesection{\normalsize Appendix}
\section{ }

\renewcommand\thesection{A}
\normalsize

\subsection{}
\begin{lemma}\label{my7963}
 Let $(Y,Z,K)$ be a $\mathcal{S}^{\infty}\times\mathcal{H}^2\times\mathcal{A}^2$-solution to the reflected BSDEs \eqref{my3}. Assume the driver $g$ satisfies Assumption \emph{(H2)}. Then, $(Z,K)\in BMO\times\mathcal{A}.$
\end{lemma}
\begin{proof}
Applying  It\^{o}'s formula to $e^{-2\gamma Y_t}$ yields that for any $\tau\in \mathcal{T}_{0}$
	\begin{align*}
	&2\gamma^2\int^T_\tau e^{-2\gamma Y_s}|Z_s|^2ds \leq  e^{-2\gamma \eta} -2\gamma\int^T_\tau e^{-2\gamma Y_s}g(s,Y_s,Z_s)ds+2\gamma\int^T_\tau e^{-2\gamma Y_s}(Z_sdB_s-dK_s)\\
	\ \ \ & \leq  e^{-2\gamma \eta}+2\gamma\int^T_\tau e^{-2\gamma Y_s}\bigg(\alpha+\beta|Y_s|+\frac{\gamma}{2}|Z_s|^2\bigg)ds
	+2\gamma\int^T_\tau e^{-2\gamma Y_s}Z_sdB_s,
	\end{align*}
	where we used the fact that $K$ is a non-decreasing process in the last inequality. Thus, we could derive that
	\begin{align*}
	\gamma^2e^{-2\gamma\|Y\|_{\mathcal{S}^{\infty}}}\int^T_\tau |Z_s|^2ds &\leq \gamma^2\int^T_\tau e^{-2\gamma Y_s}|Z_s|^2ds\\
	&\leq (1+2\gamma T(\alpha+\beta\|Y\|_{\mathcal{S}^{\infty}})) e^{2\gamma\|Y\|_{\mathcal{S}^{\infty}}}+ \gamma\int^T_\tau e^{-2\gamma Y_s}Z_sdB_s,
	\end{align*}
	which implies that $Z\in BMO$. In particular $Z\in\mathcal{H}^{d}$. Then, by a standard calculus, we could get that $K\in\mathcal{A}$, which ends the proof.
\end{proof}
\medskip

{Let us now turn to the proofs of Lemma \ref{myq7901} and Lemma \ref{myq7902}. The main idea is the same as in Lemma \ref{myq7900} (see also that of \cite[Theorem 2.8]{FHT2}). For the reader's convenience, we shall give the sketch of these proofs.}
\medskip

\subsection{Proof of Lemma \ref{myq7901}}

It follows from Lemma \ref{my6} and Remark \ref{my900} that for any $m\geq 1$
\begin{align}\label{my100}
Y^{(m)}_t:=\esssup_{\tau\in\mathcal{T}_{t}} y^{(m),\tau}_t, \ \ \forall t\in[0,T],
\end{align}
in which $y^{(m),\tau}_t$ is the solution of the following quadratic BSDE
{\small \begin{align*}
y^{(m),\tau}_t=\xi\mathbf{1}_{\{\tau=T\}}+h (\tau,Y^{(m-1)}_{\tau},{(\mathbf{P}_{Y^{(m-1)}_s})_{s=\tau}}) \mathbf{1}_{\{\tau<T\}}+\int^{\tau}_t f(s,y^{(m),\tau}_s,\mathbf{P}_{Y^{(m-1)}_s},z^{(m),\tau}_s)ds-\int^{\tau}_t z^{(m),\tau}_s dB_s.
\end{align*}
}

\noindent Thanks to assertion (i) of Lemma \ref{my7} (taking  $\alpha_t=\alpha+\beta\mathbf{E}[|Y^{(m-1)}_t|]$) and in view of \eqref{myq82}, we get for any $t\in[0,T]$,
\begin{align}	\label{myq523}
	\begin{split}
	&\exp\left\{{\gamma}\big|y^{(m),\tau}_t\big|\right\}\\
	&\leq \mathbf{E}_t\bigg[\exp\bigg\{\gamma e^{\beta (T-t)}\big(|\xi|+\eta+\gamma_1\sup\limits_{s\in[t,T]}|Y^{(m-1)}_{s}|+(\gamma_2+\beta(T-t))\sup\limits_{s\in[t,T]}\mathbf{E}[|Y^{(m-1)}_{s}|]\big)\bigg\}	\bigg],
	\end{split}
\end{align}
in which $\eta=\alpha T+\sup\limits_{s\in[0,T]}|h(s,0,\delta_0)|.$
Recalling \eqref{my100} and applying {Doob's maximal inequality} and  Jensen’s inequality, we get  that for each  $m\geq 1, p\geq 2 $ and $t\in[0,T]$
\begin{align*}	
	\begin{split}
	&\mathbf{E}\bigg[\exp\big\{{p\gamma}\sup\limits_{s\in[t,T]}\big|Y^{(m)}_s\big|\big\}\bigg]\leq 4 \mathbf{E}\bigg[\exp\bigg\{p\gamma e^{\beta (T-t)}\big(|\xi|+\eta+\gamma_1\sup\limits_{s\in[t,T]}|Y^{(m-1)}_{s}|\big)\bigg\}	 \bigg]\\
	& \ \ \ \ \ \  \ \ \ \ \ \  \ \ \ \ \ \ \ \ \ \ \ \ \ \ \ \ \ \ \ \ \ \ \  \ \ \ \ \ \  \ \ \ \ \ \  \times\mathbf{E}\bigg[\exp\bigg\{p\gamma e^{\beta (T-t)}(\gamma_2+\beta(T-t))\sup\limits_{s\in[t,T]}|Y^{(m-1)}_{s}|\bigg\}	 \bigg].
	\end{split}
\end{align*}
Under assumption \eqref{myq950}, we can then find three constants $h\in (0,T]$ and $\nu, \widetilde{\nu}>1$  depending only on $\beta,\gamma_1$ and $\gamma_2$ such that
\begin{align}\label{myq7906}
4e^{\beta h}\widetilde{\nu}(\gamma_1+\gamma_2+\beta h) < 1 ~~ \text{and}~~ 4e^{\beta h}\nu\widetilde{\nu} \gamma_1<1.\end{align}
In the spirit of H\"{o}lder's inequality, we derive that for any  $p\geq 2$
 {\small \begin{align*}
	\begin{split}
 &\mathbf{E}\bigg[\exp\big\{{p\gamma}\sup\limits_{s\in[T-h,T]}\big|Y^{(m)}_s\big|\big\}\bigg]
\\
&\leq 4 \mathbf{E}\bigg[\exp\bigg\{\frac{\nu p\gamma}{\nu-1}e^{\beta h}(|\xi|+\eta)\bigg\}\bigg]^{\frac{\nu-1}{\nu}}\mathbf{E}\bigg[\exp\bigg\{p\gamma\sup\limits_{s\in[T-h,T]}|Y^{(m-1)}_{s}|\bigg\}	\bigg]^{e^{\beta h}(\gamma_1+\gamma_2+\beta h)}\\
&\leq 4\mathbf{E}\bigg[\exp\bigg\{\frac{2\nu p\gamma}{\nu-1}e^{\beta h}|\xi|\bigg\}\bigg]^{\frac{\nu-1}{2\nu}}\mathbf{E}\bigg[\exp\bigg\{\frac{2\nu p\gamma}{\nu-1}e^{\beta h}\eta\bigg\}\bigg]^{\frac{\nu-1}{2\nu}}  \mathbf{E}\bigg[\exp\bigg\{p\gamma\sup\limits_{s\in[T-h,T]}|Y^{(m-1)}_{s}|\bigg\}	\bigg]^{e^{\beta h}(\gamma_1+\gamma_2+\beta h)}.
	\end{split}
\end{align*}}

\noindent Define $ \rho = \frac{1}{1-e^{\beta h}(\gamma_1+\gamma_2+\beta h)}$ and \begin{align*}
\mu:=
\begin{cases} \frac{T}{h}, \  &\text{if $\frac{T}{h}$ is an integer};\\
 [\frac{T}{h}]+1, \ &\text{otherwise}.
\end{cases}
\end{align*}
If $\mu=1$, it follows from the previous inequality that for each $p\geq 2$ and $m\geq 1$
\begin{align*}
\begin{split}
&\mathbf{E}\bigg[\exp\bigg\{p\gamma\sup\limits_{s\in[0,T]}\big|Y^{(m)}_s\big|\bigg\}
\bigg] \\
&\leq 4\mathbf{E}\bigg[\exp\bigg\{\frac{2\nu p\gamma}{\nu-1} e^{\beta h}|\xi|\bigg\}\bigg]^{\frac{1}{2}}\mathbf{E}\bigg[\exp\bigg\{\frac{2\nu p\gamma}{\nu-1} e^{\beta h}\eta\bigg\}\bigg]^{\frac{1}{2}} \mathbf{E}\bigg[\exp\bigg\{{p\gamma} \sup\limits_{s\in[0,T]}|Y^{(m-1)}_s|\bigg\}	 \bigg]^{e^{\beta h}(\gamma_1+\gamma_2+\beta h)}.
\end{split}
\end{align*}
Iterating the above procedure $m$ times, we get,
{\small \begin{align}\label{myq698}
\begin{split}
&\mathbf{E}\bigg[\exp\bigg\{{p\gamma}\sup\limits_{s\in[0,T]}\big|Y^{(m)}_s\big|\bigg\}
\bigg]\\
& \leq 4^{\rho}\mathbf{E}\bigg[\exp\bigg\{\frac{2\nu p\gamma}{\nu-1}  e^{\beta h} |\xi|\bigg\}\bigg]^{\frac{\rho}{2}}\mathbf{E}\bigg[\exp\bigg\{\frac{2\nu p\gamma}{\nu-1}  e^{\beta h}\eta\bigg\}\bigg]^{\frac{\rho}{2}} \mathbf{E}\bigg[\exp\bigg\{p\gamma \sup\limits_{s\in[0,T]}|Y^{(0)}_s|\bigg\}	\bigg]^{e^{m\beta h}(\gamma_1+\gamma_2+\beta h)^m},
\end{split}
\end{align}}

\noindent which is uniformly bounded with respect to $m$. If $\mu=2$, proceeding identically as in the above, we have for any $p\geq 2$,
{\small \begin{align}\label{myq699}
 \begin{split}
&\mathbf{E}\bigg[\exp\bigg\{{p\gamma}\sup\limits_{s\in[T-h,T]}\big|Y^{(m)}_s\big|\bigg\}
\bigg]\\
& \leq 4^{\rho}\mathbf{E}\bigg[\exp\bigg\{\frac{2\nu p\gamma}{\nu-1} e^{\beta h}|\xi|\bigg\}\bigg]^{\frac{\rho}{2}}\mathbf{E}\bigg[\exp\bigg\{\frac{2\nu p\gamma}{\nu-1}  e^{\beta h}\eta\bigg\}\bigg]^{\frac{\rho}{2}}  \mathbf{E}\bigg[\exp\bigg\{p\gamma \sup\limits_{s\in[0,T]}|Y^{(0)}_s|\bigg\}	\bigg]^{e^{m\beta h}(\gamma_1+\gamma_2+\beta h)^m}.
\end{split}
\end{align}
}
Then, consider the following quadratic reflected BSDEs  on time interval $[0,T-h]$:
{\small \begin{align*}
\begin{cases}
&Y_t^{(m)}=Y_{T-h}^{(m)}+\int_t^{T-h} f(s,Y^{(m-1)}_s,\mathbf{P}_{Y^{(m-1)}_s},Z^{(m)}_s)ds-\int_t^{T-h} Z^{(m)}_sdB_s+K^{(m)}_T-K^{(m)}_t, \quad 0\le t\le T-h,\\
& Y^{(m)}_t\geq h(t,Y^{(m-1)}_t, \mathbf{P}_{Y^{(m-1)}_t}),\quad \forall t\in [0,T-h] \,\,\mbox{ and }\,\, \int_0^{T-h} (Y^{(m)}_t-h(t,Y^{(m-1)}_t,\mathbf{P}_{Y^{(m-1)}_t}))dK^{(m)}_t = 0.
\end{cases}
\end{align*}}
In view of the derivation of \eqref{myq698}, we deduce that
{\small \begin{align*}
	\begin{split}
	&\mathbf{E}\bigg[\exp\bigg\{{p\gamma}\sup\limits_{s\in[0,T-h]}\big|Y^{(m)}_s\big|\bigg\}
	\bigg]\\ &\leq 4^{\rho}\mathbf{E}\bigg[\exp\bigg\{\frac{2\nu p\gamma}{\nu-1}  e^{\beta h}\big|Y_{T-h}^{(m)}\big|\bigg\}\bigg]^{\frac{\rho}{2}}\mathbf{E}\bigg[\exp\bigg\{\frac{2\nu p\gamma}{\nu-1}  e^{\beta h}\eta\bigg\}\bigg]^{\frac{\rho}{2}} \mathbf{E}\bigg[\exp\bigg\{p\gamma \sup\limits_{s\in[0,T]}|Y^{(0)}_s|\bigg\}	\bigg]^{e^{m\beta h}(\gamma_1+\gamma_2+\beta h)^m}\\
	&\leq 4^{\rho+\frac{\rho^2}{2}}\mathbf{E}\bigg[\exp\bigg\{\bigg(\frac{2\nu e^{\beta h}}{\nu-1}\bigg)^2 p\gamma|\xi|\bigg\}\bigg]^{\frac{\rho^2}{4}}\mathbf{E}\bigg[\exp\bigg\{\bigg(\frac{2\nu e^{\beta h}}{\nu-1}\bigg)^2 p\gamma\eta\bigg\}\bigg]^{\frac{\rho^2}{4}}\mathbf{E}\bigg[\exp\bigg\{\frac{2\nu p\gamma}{\nu-1} e^{\beta h}\eta\bigg\}\bigg]^{\frac{\rho}{2}} \\
	&\ \ \ \ \ \ \times\mathbf{E}\bigg[\exp\bigg\{\frac{2\nu p\gamma}{\nu-1} e^{\beta h}\sup\limits_{s\in[0,T]}|Y^{(0)}_s|\bigg\}	\bigg]^{\frac{\rho}{2}e^{m\beta h}(\gamma_1+\gamma_2+\beta h)^m} \mathbf{E}\bigg[\exp\bigg\{p\gamma \sup\limits_{s\in[0,T]}|Y^{(0)}_s|\bigg\}	\bigg]^{e^{m\beta h}(\gamma_1+\gamma_2+\beta h)^m},
	\end{split}
	\end{align*}
}where we used \eqref{myq699} in the last inequality.
Putting the above inequalities together and applying H\"{o}lder's inequality again yields that for any $p\geq 2$
{\small \begin{align}\label{myq516}
	\begin{split}
	&\mathbf{E}\bigg[\exp\bigg\{{p\gamma}\sup\limits_{s\in[0,T]}\big|Y^{(m)}_s\big|\bigg\}
	\bigg]
	\leq 	\mathbf{E}\bigg[\exp\bigg\{{2p\gamma}\sup\limits_{s\in[0,T-h]}|Y^{(m)}_s|\bigg\}
	\bigg]^{\frac{1}{2}}	\mathbf{E}\bigg[\exp\bigg\{{2p\gamma}\sup\limits_{s\in[T-h,T]}|Y^{(m)}_s|\bigg\}
	\bigg]^{\frac{1}{2}}\\
	&\leq 4^{\rho+\frac{\rho^2}{4}}\mathbf{E}\bigg[\exp\bigg\{\bigg(\frac{2\nu e^{\beta h}}{\nu-1}\bigg)^2 2p\gamma|\xi|\bigg\}\bigg]^{\frac{\rho}{4}+\frac{\rho^2}{8}}\mathbf{E}\bigg[\exp\bigg\{\bigg(\frac{2\nu e^{\beta h}}{\nu-1}\bigg)^2 2p\gamma\eta\bigg\}\bigg]^{\frac{\rho}{2}+\frac{\rho^2}{8}}\\
	&\hspace*{4cm} \times\mathbf{E}\bigg[\exp\bigg\{\frac{4\nu}{\nu-1}p\gamma e^{\beta h} \sup\limits_{s\in[0,T]}|Y^{(0)}_s|\bigg\}	\bigg]^{(1+\frac{\rho}{4})e^{m\beta h}(\gamma_1+\gamma_2+\beta h)^m},
	\end{split}
	\end{align}
}

\noindent which is uniformly bounded with respect to $m$.
Iterating the above procedure $\mu$ times in the general case and recalling \cite[Theorem 3.2]{BY}, we eventually get
\begin{align*}
	\begin{split}
\sup\limits_{m\geq 0}\mathbf{E}\bigg[\exp\bigg\{p\gamma\sup\limits_{s\in[0,T]}|Y^{(m)}_s|\bigg\}+\bigg(\int^T_0|Z^{(m)}_t|^2dt\bigg)^p+|K^{(m)}_T|^p \bigg]<\infty, \ \forall p\geq 1,
	\end{split}
\end{align*}
which concludes the proof.

\medskip

\subsection{Proof of Lemma \ref{myq7902}}

Without loss of generality, assume $f(t,y,v,\cdot)$ is concave, since the other case can be proved in a similar way, see Remark \ref{myrk11}.
For each fixed $m,q\geq 1$ and $\theta\in(0,1)$, we can define similarly $\delta_{\theta}\ell^{(m,q)}$ and  $\delta_{\theta}\widetilde{\ell}^{(m,q)}$ for $y^{\tau},z^{\tau}$.
Then, the pair of processes $(\delta_{\theta}y^{(m,q),\tau},\delta_{\theta}z^{(m,q),\tau})$
satisfies the following BSDE:
{\small \begin{align}\label{myq12}
\begin{split}
\delta_{\theta}y^{(m,q),\tau}_t=&\delta_{\theta}\eta^{(m,q),\tau}+\int^\tau_t\left(\delta_{\theta}f^{(m,q)}(s,\delta_{\theta}y^{(m,q),\tau}_s,\delta_{\theta}z^{(m,q),\tau}_s)+\delta_{\theta}f_0^{(m,q)}(s)\right)ds-\int^\tau_t\delta_{\theta}z^{(m,q),\tau}_sdB_s,
\end{split}
\end{align}}
where the terminal condition and generator are given by
\begin{align*}
&\delta_{\theta}\eta^{(m,q)}=-\xi\mathbf{1}_{\{\tau=T\}}+\frac{\theta h (\tau,Y^{(m+q-1)}_{\tau},{(\mathbf{P}_{Y^{(m+q-1)}_s})_{s=\tau}})- h (\tau,Y^{(m-1)}_{\tau},{(\mathbf{P}_{Y^{(m-1)}_s})_{s=\tau}})}{1-\theta} \mathbf{1}_{\{\tau<T\}},\\
& \delta_{\theta}f_0^{(m,q)}(t)=
\frac{1}{1-\theta}\left(f(t,y^{(m),\tau}_{t},\mathbf{P}_{Y^{(m+q-1)}_t}, z^{(m),\tau}_t)-f(t,y^{(m),\tau}_{t},\mathbf{P}_{Y^{(m-1)}_t}, z^{(m),\tau}_t)\right),\\
&
\delta_{\theta}f^{(m,q)}(t,y,z)=\frac{1}{1-\theta}\bigg(
\theta f(t,y^{(m+q),\tau}_{t},\mathbf{P}_{Y^{(m+q-1)}_t}, z^{(m+q),\tau}_t)\\
&\ \ \ \ \ \ \ \ \ \ \ \ \ \ \ \ \ \ \ \  \ \  \ \ \ \ \ \ \ \ \ \ \ \ \ \ \ \ - f(t,-(1-\theta)y+\theta y^{(m+q),\tau}_t,\mathbf{P}_{Y^{(m+q-1)}_t}, -(1-\theta)z+\theta z^{(m+q),\tau}_t)\bigg).
\end{align*}

\noindent Recalling Assumptions (H2), (H4), (H5') and (H6), we have
\begin{align*}
&\delta_{\theta}\eta^{(m,q)}\leq |\xi|+|h(\tau,0,\delta_0)|+\gamma_1 \left( 2|Y^{(m+q-1)}_{\tau}|+|\delta_{\theta}Y^{(m-1,q)}_\tau|\right)\\
&\hspace*{5cm} +\gamma_2\left(2\mathbf{E}[|Y^{(m+q-1)}_{s}|]_{s=\tau}+\mathbf{E}[|\delta_{\theta}Y^{(m-1,q)}_s|]_{s=\tau})\right),\\
&\delta_{\theta}f_0^{(m,q)}(t)\leq \beta \left(\mathbf{E}[|Y^{(m+q-1)}_{t}|]+\mathbf{E}[|\delta_{\theta}Y^{(m-1,q)}_t|]\right),\\
&\delta_{\theta}f^{(m,q)}(t,y,z)
 \leq \beta|y|+\beta |y^{(m+q),\tau}_t| -f(t,y^{(m+q),\tau}_t,\mathbf{P}_{Y^{(m+q-1)}_t}, -z)
\\
&\hspace*{5cm} \leq \alpha+2\beta |y^{(m+q),\tau}_t|+\beta\mathbf{E}[|Y^{(m+q-1)}_{t}|]+\beta|y|+\frac{\gamma}{2}|z|^2.
\end{align*}
For any $m,q\geq 1$, set $ C_2:=\sup\limits_{m}\mathbf{E}\big[\sup\limits_{s\in[0,T]}|Y^{(m)}_{s}|\big]<\infty$ and
\begin{align*}
&\zeta^{(m,q)}= e^{\beta T}\bigg(|\xi|+\eta+\gamma_1\big(\sup\limits_{s\in[0,T]}|Y^{(m-1)}_{s}|+\sup\limits_{s\in[0,T]}|Y^{(m+q-1)}_{s}|\big)+(\gamma_2+\beta T)C_2\bigg),\\
& \chi^{(m,q)}:=\eta+2 \gamma_1 \bigg(\sup\limits_{s\in[0,T]}|Y^{(m+q-1)}_{s}|+\sup\limits_{s\in[0,T]}|Y^{(m-1)}_{s}|\bigg)+2(\gamma_2+\beta T)C_2,\\
&\widetilde{\chi}^{(m,q)}:=\chi^{(m,q)}+\sup\limits_{s\in[0,T]}|Y^{(m+q)}_{s}|+\sup\limits_{s\in[0,T]}|Y^{(m)}_{s}|.\end{align*}

\noindent Using assertion (ii) of Lemma \ref{my7} to  \eqref{myq12} and
H\"{o}lder's inequality, we derive that  for any $p\geq 1$
{\small  \begin{align*}	
 \begin{split}
\exp\left\{{p\gamma}\big(\delta_{\theta}y^{(m,q),\tau}_t\big)^+\right\} &\leq   \mathbf{E}_t\bigg[\exp\bigg\{p\gamma e^{\beta(T-t)}\bigg(|\xi|+\chi^{(m,q)}+2\beta T\sup\limits_{s\in[t,T]}|y^{(m+q),\tau}_{s}| +\gamma_1\sup\limits_{s\in[t,T]}|\delta_{\theta}Y^{(m-1,q)}_{s}|
\\&\hspace*{5cm} +(\gamma_2+\beta (T-t))\sup\limits_{s\in[t,T]}\mathbf{E}[|\delta_{\theta}Y^{(m-1,q)}_{s}|] \bigg)\bigg\}	\bigg]\\
 &\leq   \mathbf{E}_t\bigg[\exp\bigg\{2p\gamma e^{\beta(T-t)}\bigg(|\xi|+\chi^{(m,q)} +\gamma_1\sup\limits_{s\in[t,T]}|\delta_{\theta}Y^{(m-1,q)}_{s}|\\&\hspace*{0.5cm} +(\gamma_2+\beta (T-t))\sup\limits_{s\in[t,T]}\mathbf{E}[|\delta_{\theta}Y^{(m-1,q)}_{s}|] \bigg)\bigg\}	\bigg]^{\frac{1}{2}} \mathbf{E}_t\bigg[\exp\bigg\{4 p\gamma e^{2\beta T}\sup\limits_{s\in[t,T]}|y^{(m+q),\tau}_{s}| \bigg\}	 \bigg]^{\frac{1}{2}}.
 \end{split}
	\end{align*}
}

\noindent Recalling \eqref{myq523} and using {Doob's maximal inequality}, we conclude that for each $p\geq 2$ and $t\in[0,T]$
\begin{align*}
\mathbf{E}_t\left[\exp\big\{{p\gamma}\sup\limits_{s\in[t,T]}\big|y^{(m),\tau}_s\big|\big\}\right]\vee \mathbf{E}_t\left[\exp\big\{{p\gamma}\sup\limits_{s\in[t,T]}\big|y^{(m+q),\tau}_s\big|\big\}\right]\leq 4 \mathbf{E}_t\left[\exp\{{p\gamma} \zeta^{(m,q)}\} \right], \ \forall m,q\geq 1.
\end{align*}
It follows from \eqref{my100}  that
{\small  \begin{align*}
 \begin{split}
 &\exp\left\{{p\gamma}\big(\delta_{\theta}Y^{(m,q)}_t\big)^+\right\}\leq \esssup_{\tau\in\mathcal{T}_{t}}\exp\left\{{p\gamma}\big(\delta_{\theta}y^{(m,q),\tau}_t\big)^+\right\}\\
 &\leq 4 \mathbf{E}_t\bigg[\exp\bigg\{2p\gamma e^{\beta(T-t)}\bigg(|\xi|+\chi^{(m,q)} +\gamma_1\sup\limits_{s\in[t,T]}|\delta_{\theta}Y^{(m-1,q)}_{s}| +(\gamma_2+\beta (T-t))\sup\limits_{s\in[t,T]}\mathbf{E}[|\delta_{\theta}Y^{(m-1,q)}_{s}|] \bigg)\bigg\}	\bigg]^{\frac{1}{2}}\\
 &\ \ \ \ \ \ \ \ \  \times  \mathbf{E}_t\bigg[\exp\bigg\{4 p\gamma e^{2\beta T}\zeta^{(m,q)} \bigg\}	\bigg]^{\frac{1}{2}}.
 \end{split}
	\end{align*}
}

\noindent Using a similar method, we derive that
{\small  \begin{align*}	
 \begin{split}
 &\exp\left\{{p\gamma}\big(\delta_{\theta}\widetilde{Y}^{(m,q)}_t\big)^+\right\}\\
 &\leq  4  \mathbf{E}_t\bigg[\exp\bigg\{2p\gamma e^{\beta(T-t)}\bigg(|\xi|+\chi^{(m,q)} +\gamma_1\sup\limits_{s\in[t,T]}|\delta_{\theta}\widetilde{Y}^{(m-1,q)}_{s}| +(\gamma_2+\beta (T-t))\sup\limits_{s\in[t,T]}\mathbf{E}[|\delta_{\theta}\widetilde{Y}^{(m-1,q)}_{s}|] \bigg)\bigg\}	\bigg]^{\frac{1}{2}}\\
 &\ \ \ \ \ \ \ \ \  \times  \mathbf{E}_t\bigg[\exp\bigg\{4 p\gamma e^{2\beta T}\zeta^{(m,q)} \bigg\}	\bigg]^{\frac{1}{2}}.
 \end{split}
	\end{align*}
}

\noindent According to  the fact that
\begin{align*}
\big(\delta_{\theta}{Y}^{(m,q)}\big)^-
\leq \big(\delta_{\theta}\widetilde{Y}^{(m,q)}\big)^++2|Y^{(m)}| \ \text{and}\ \big(\delta_{\theta}\widetilde{Y}^{(m,q)}\big)^-
\leq \big(\delta_{\theta}{Y}^{(m,q)}\big)^++2|Y^{(m+q)}|,
\end{align*}
we deduce that
{\small  \begin{align*}
	\begin{split}
	&\exp\left\{p\gamma \big|\delta_{\theta}{Y}^{(m,q)}_t\big|\right\}\vee \exp\left\{p\gamma \big|\delta_{\theta}\widetilde{Y}^{(m,q)}_t\big|\right\}\leq
	\exp\left\{{p\gamma}\big(\big( \delta_{\theta}{Y}^{(m,q)}_t\big)^++\big(\delta_{\theta}\widetilde{Y}^{(m,q)}_t\big)^++2|Y^{(m)}_t|+2|Y^{(m+q)}_t|\big)\right \}\\
	&\leq   4^2  \mathbf{E}_t\bigg[\exp\bigg\{2p\gamma e^{\beta(T-t)}\bigg(|\xi|+\widetilde{\chi}^{(m,q)} +\gamma_1\sup\limits_{s\in[t,T]}\delta_{\theta}\overline{Y}^{(m-1,q)}_{s} +(\gamma_2+\beta (T-t))\sup\limits_{s\in[t,T]}\mathbf{E}[\delta_{\theta}\overline{Y}^{(m-1,q)}_{s}] \bigg)\bigg\}	\bigg] \\
 &\ \ \ \ \ \ \ \ \  \ \ \times \mathbf{E}_t\big[\exp\big\{4 p\gamma e^{2\beta T}\zeta^{(m,q)} \big\}	\big].
	\end{split}
	\end{align*}
}

\noindent Applying {Doob's maximal inequality} and
H\"{o}lder's inequality, we get  that for each  $p> 1$ and $t\in[0,T]$
{\small  \begin{align*}
	\begin{split}
	&\mathbf{{E}}\bigg[\exp\big\{p\gamma \sup\limits_{s\in[t,T]}\delta_{\theta}\overline{Y}^{(m,q)}_s\big\}\bigg]
	\\ &\leq  4^5  \mathbf{E}\bigg[\exp\bigg\{4p\gamma e^{\beta(T-t)}\widetilde{\nu}\bigg(|\xi|+\widetilde{\chi}^{(m,q)} +\gamma_1\sup\limits_{s\in[t,T]}\delta_{\theta}\overline{Y}^{(m-1,q)}_{s} +(\gamma_2+\beta (T-t))\sup\limits_{s\in[t,T]}\mathbf{E}[\delta_{\theta}\overline{Y}^{(m-1,q)}_{s}] \bigg)\bigg\}	\bigg]^{\frac{1}{\widetilde{\nu}}} \\
 &\ \ \ \ \ \ \ \ \  \ \ \times \mathbf{E}\bigg[\exp\bigg\{\frac{4\widetilde{\nu}}{\widetilde{\nu}-1} p\gamma e^{2\beta T}\zeta^{(m,q)} \bigg\}	\bigg]^{\frac{\widetilde{\nu}-1}{\widetilde{\nu}}}.
	\end{split}
	\end{align*}
}

\noindent Recalling the definitions of $h$, $\nu$ and $\widetilde{\nu}$ in \eqref{myq7906}, we have
{ \small \begin{align}	\label{myq512}
	\begin{split}
	&\mathbf{{E}}\bigg[\exp\big\{p\gamma \sup\limits_{s\in[T-h,T]}\delta_{\theta}\overline{Y}^{(m,q)}_s\big\}\bigg]
	\\ &\leq 4^5\mathbf{{E}}\bigg[\exp\bigg\{\frac{8\nu\widetilde{\nu} p\gamma}{\nu-1}e^{\beta h}|\xi|\bigg\}		\bigg]^{\frac{\nu-1}{2\nu \widetilde{\nu}}}\mathbf{{E}}\bigg[\exp\bigg\{\frac{8\nu\widetilde{\nu} p\gamma}{\nu-1}e^{\beta h}\widetilde{\chi}^{(m,q)}\bigg\}		\bigg]^{\frac{\nu-1}{2\nu\widetilde{\nu}}}\mathbf{E}\bigg[\exp\bigg\{\frac{4\widetilde{\nu}}{\widetilde{\nu}-1} p\gamma e^{2\beta T}\zeta^{(m,q)} \bigg\}	 \bigg]^{\frac{\widetilde{\nu}-1}{\widetilde{\nu}}}\\
	&\ \ \ \ \ \ \ \ \ \  \times \mathbf{{E}}\bigg[\exp\bigg\{p\gamma\sup\limits_{s\in[T-h,T]}\delta_{\theta}\overline{Y}^{(m-1,q)}_{s}\bigg\}		\bigg]^{4e^{\beta h}(\gamma_1+\gamma_2+\beta h)}.
	\end{split}
	\end{align}
}

\noindent Set $ \widetilde{\rho} = \frac{1}{1-4e^{\beta h}(\gamma_1+\gamma_2+\beta h)}$. If $\mu=1$, it follows from \eqref{myq512} that for each $p\geq 1$ and $m,q\geq 1$
{\small \begin{align*}
\begin{split}
&\mathbf{E}\bigg[\exp\big\{p\gamma\sup\limits_{s\in[0,T]}\delta_{\theta}\overline{Y}^{(m,q)}_s\big\}
\bigg]
\\ &\leq 4^{5\widetilde{\rho}}\mathbf{{E}}\bigg[\exp\bigg\{\frac{8\nu\widetilde{\nu} p\gamma}{\nu-1}e^{\beta h}|\xi|\bigg\}		\bigg]^{\frac{\widetilde{\rho}}{2}}\sup\limits_{m,q\geq 1}\mathbf{{E}}\bigg[\exp\bigg\{\frac{8\nu\widetilde{\nu} p\gamma}{\nu-1}e^{\beta h}\widetilde{\chi}^{(m,q)}\bigg\}		\bigg]^{\frac{\widetilde{\rho}}{2}}\sup\limits_{m,q\geq 1}\mathbf{E}\bigg[\exp\bigg\{\frac{4\widetilde{\nu}}{\widetilde{\nu}-1} p\gamma e^{2\beta T}\zeta^{(m,q)} \bigg\}	\bigg]^{{\widetilde{\rho}}}\\
	&\ \ \ \ \ \ \ \ \ \  \times \mathbf{{E}}\bigg[\exp\bigg\{p\gamma\sup\limits_{s\in[0,T]}\delta_{\theta}\overline{Y}^{(1,q)}_{s}\bigg\}		\bigg]^{e^{(m-1)\beta h}(4\gamma_1+4\gamma_2+4\beta h)^{m-1}}.
\end{split}
\end{align*}}

\noindent Applying Lemma \ref{myq7901}, we have for any $\theta\in(0,1)$
\[
\lim_{m\rightarrow \infty}\sup\limits_{q\geq 1}\mathbf{{E}}\bigg[\exp\bigg\{p\gamma\sup\limits_{s\in[0,T]}\delta_{\theta}\overline{Y}^{(1,q)}_{s}\bigg\}		\bigg]^{e^{(m-1)\beta h}(4\gamma_1+4\gamma_2+4\beta h)^{m-1}}=1.
\]
It follows that
\begin{align*}
&\sup\limits_{\theta\in(0,1)}\lim_{m\rightarrow \infty}\sup\limits_{q\geq 1}\mathbf{E}\bigg[\exp\big\{p\gamma\sup\limits_{s\in[0,T]}\delta_{\theta}\overline{Y}^{(m,q)}_s\big\}
\bigg] \\
&\hspace*{1cm}\leq 4^{5\widetilde{\rho}}\mathbf{{E}}\bigg[\exp\bigg\{\frac{8\nu\widetilde{\nu} p\gamma}{\nu-1}e^{\beta h}|\xi|\bigg\}		\bigg]^{\frac{\widetilde{\rho}}{2}}\sup\limits_{m,q\geq 1}\mathbf{{E}}\bigg[\exp\bigg\{\frac{8\nu\widetilde{\nu} p\gamma}{\nu-1}e^{\beta h}\widetilde{\chi}^{(m,q)}\bigg\}		\bigg]^{\frac{\widetilde{\rho}}{2}} \\ &\hspace*{6cm} \times \sup\limits_{m,q\geq 1}\mathbf{E}\bigg[\exp\bigg\{\frac{4\widetilde{\nu}}{\widetilde{\nu}-1} p\gamma e^{2\beta T}\zeta^{(m,q)} \bigg\}	\bigg]^{{\widetilde{\rho}}}<\infty.
\end{align*}
If $\mu=2$, proceeding identically as to derive \eqref{myq516}, we have for any $p\geq 1$
 \begin{align*}
	\begin{split}
	&\mathbf{E}\bigg[\exp\big\{{p\gamma}\sup\limits_{s\in[0,T]}\delta_{\theta}\overline{Y}^{(m,q)}_s\big\}
	\bigg]\\
	&\leq 4^{5\widetilde{\rho}+\frac{5\widetilde{\rho}^2}{4}}\mathbf{E}\bigg[\exp\bigg\{\bigg(\frac{8\nu\widetilde{\nu}e^{\beta h}}{\nu-1}\bigg)^2 2p\gamma|\xi|\bigg\}\bigg]^{\frac{\widetilde{\rho}}{4}+\frac{\widetilde{\rho}^2}{8}}\sup\limits_{m,q\geq 1}\mathbf{E}\bigg[\exp\bigg\{\bigg(\frac{8\nu\widetilde{\nu}e^{\beta h}}{\nu-1}\bigg)^2 2p\gamma\widetilde{\chi}^{(m,q)}\bigg\}\bigg]^{\frac{\widetilde{\rho}}{2}+\frac{\widetilde{\rho}^2}{8}}\\
	&\ \ \ \ \ \ \ \ \ \ \ \ \ \times\sup\limits_{m,q\geq 1}\mathbf{E}\bigg[\exp\bigg\{\frac{32\nu\widetilde{\nu}^2 e^{\beta h}}{(\widetilde{\nu}-1)({\nu-1})} p\gamma e^{2\beta T}\zeta^{(m,q)} \bigg\}	 \bigg]^{{\widetilde{\rho}+\frac{\widetilde{\rho}^2}{4}}}\\
	&\ \ \ \ \ \ \ \ \ \ \ \ \ \times\mathbf{E}\bigg[\exp\bigg\{\frac{8\nu\widetilde{\nu}e^{\beta h}}{\nu-1}p\gamma \sup\limits_{s\in[0,T]}|\delta_{\theta}\overline{Y}^{(1,q)}_s|\bigg\}	 \bigg]^{(1+\frac{\widetilde{\rho}}{4})e^{(m-1)\beta h}(4\gamma_1+4\gamma_2+4\beta h)^{m-1}},
	\end{split}
	\end{align*}
which also implies the desired assertion in this case. Iterating the above procedure $\mu$ times in the general case, we get the desired result.

\end{document}